\documentclass[reqno,a4paper, 11pt, twoside]{amsart}
\usepackage{amsmath}
\usepackage{amscd}
\usepackage{amssymb}
\usepackage{amsthm}
\usepackage{graphics}
\usepackage{indentfirst}
\usepackage{latexsym}
\usepackage{amsfonts}
\usepackage{multicol}
\usepackage{stmaryrd}
\usepackage{array}
\usepackage{pxfonts}
\usepackage{verbatim}
\usepackage{bbm}
\usepackage{pgf}
\usepackage{tikz}
\usepackage{nccmath}
\usepackage{needspace}
\usepackage[a4paper,margin=1in,includehead,includefoot]{geometry}

\makeatletter
\renewenvironment{proof}[1][\proofname] {\par\pushQED{\qed}\normalfont\topsep6\p@\@plus6\p@\relax\trivlist\item[\hskip\labelsep\bfseries#1\@addpunct{.}]\ignorespaces}{\popQED\endtrivlist\@endpefalse}
\makeatother

\newtheoremstyle{mattthm}{}{}{\itshape}{}{\bfseries}{.}{ }{}

\theoremstyle{mattthm}
\newtheorem{lemma}{Lemma}[section]

\newtheorem{propn}[lemma]{Proposition}

\newtheorem{thm}[lemma]{Theorem}

\newtheorem{cory}[lemma]{Corollary}
\newtheorem{conj}[lemma]{Conjecture}

\newtheoremstyle{mattdef}{}{}{}{}{\bfseries}{.}{ }{}

\theoremstyle{mattdef}
\newtheorem*{rmk}{Remark}
\newtheorem*{defn}{Definition}
\newtheorem*{eg}{Example}

\begin{document}

\newenvironment{pfenum}{\begin{proof}\indent\begin{enumerate}\vspace{-\topsep}}{\qedhere\end{enumerate}\end{proof}}

\setlength\fboxsep{1pt}
\newcommand\arrowover[1]{\scriptstyle\overrightarrow{\displaystyle #1}}
\newcommand\tub[2]{_{\lan #1,#2\ran}}
\newcommand\agree[2]{\genfrac{|}{|}{0pt}{}{#1}{#2}}
\newcommand\ten{10}
\newcommand\eleven{11}
\newcommand\dd{\yo{\bxy d}}
\newcommand\ba{\yo{\bxy a}}
\newcommand\bb{\yo{\bxy b}}
\newcommand\pp{\raising{\small$\begin{tikzpicture}[scale=0.42]\draw[fill=white](0,0)rectangle(1.5,1);\draw(0.75cm,0.5cm)node[fill=white,inner sep=0.5pt]{$\vphantom1\smash{\dpo}$};\end{tikzpicture}$}}
\newcommand\mo{{-}1}
\newcommand\reg{\mathcal{R}}
\newcommand\la\lambda
\newcommand\binoq[2]{\genfrac{[}{]}{0pt}{}{#1}{#2}}
\newcommand\gauss[2]{\genfrac{[}{]}{0pt}{}{#1}{#2}}
\newcommand\mbinoq[2]{\Bigl[\medmath{\genfrac{}{}{0pt}{}{#1}{#2}}\Bigr]}
\newcommand{\lan}{\langle}
\newcommand{\ran}{\rangle}
\newcommand\cala{\mathcal{A}}
\newcommand\calb{\mathcal{B}}
\newcommand\calf{\mathcal{F}}
\newcommand\calh{\mathcal{H}}
\newcommand\cali{\mathcal{I}}
\newcommand\call{\mathcal{L}}
\newcommand\calp{\mathcal{P}}
\newcommand\calq{\mathcal{Q}}
\newcommand\cals{\mathcal{S}}
\newcommand\calt{\mathcal{T}}
\newcommand\calr{\mathcal{R}}
\newcommand\calu{\mathcal{U}}
\newcommand\calla{\mathtt{A}}
\newcommand\callb{\mathtt{B}}
\newcommand\callc{\mathtt{C}}
\newcommand{\thmlc}[3]{\textup{\textbf{(\!\! #1 \cite[#3]{#2})}}}
\newcommand{\thmlabel}[1]{\textup{\textbf{#1}}\ }
\newcommand{\dom}{\trianglerighteqslant}
\newcommand{\doms}{\vartriangleright}
\newcommand{\ndom}{\ntrianglerighteqslant}
\newcommand{\ndoms}{\not\vartriangleright}
\newcommand{\domby}{\trianglelefteqslant}
\newcommand{\domsby}{\vartriangleleft}
\newcommand{\ndomby}{\ntrianglelefteqslant}
\newcommand{\ndomsby}{\not\vartriangleleft}
\newcommand{\subs}[1]{\subsection{#1}\indent}
\newcommand{\subsubs}[1]{\subsubsection{#1}\indent}
\newcommand{\nchar}{\operatorname{char}}
\newcommand{\thmcite}[2]{\textup{\textbf{\cite[#2]{#1}}}\ }
\newcommand\sra[2]{\stackrel{#1;\,#2}{\longrightarrow}}
\newcommand\ars[2]{\stackrel{#1;\,#2}{\longleftarrow}}
\newcommand{\bbc}{\mathbb{C}}
\newcommand{\bbf}{\mathbb{F}}
\newcommand{\bbg}{\mathbb{G}}
\newcommand{\bbn}{\mathbb{N}}
\newcommand{\bbq}{\mathbb{Q}}
\newcommand{\bbr}{\mathbb{R}}
\newcommand{\bbz}{\mathbb{Z}}
\newcommand\zo{\bbn_0}
\newcommand{\gs}{\geqslant}
\newcommand{\ls}{\leqslant}
\renewcommand{\geq}{\geqslant}
\renewcommand{\leq}{\leqslant}
\newcommand{\sect}[1]{\section{#1}}
\newcommand{\clam}{\begin{description}\item[\hspace{\leftmargin}Claim.]}
\newcommand{\prof}{\item[\hspace{\leftmargin}Proof.]}
\newcommand{\malc}{\end{description}}
\newcommand\ppmod[1]{\ (\operatorname{mod}\ #1)}
\renewcommand\hom{\operatorname{Hom}}
\newcommand\ehom{\operatorname{EHom}}
\newcommand\im{\operatorname{im}}
\newcommand\cha{\operatorname{char}}
\newcommand\inc{\mathfrak{A}}
\newcommand\fsl{\mathfrak{sl}}
\newcommand\sgn{\operatorname{sgn}}
\newcommand\lra\longrightarrow
\newcommand\lexg{>_{\operatorname{lex}}}
\newcommand\lexgs{\gs_{\operatorname{lex}}}
\newcommand\lexl{<_{\operatorname{lex}}}
\newcommand\tru[1]{{#1}_-}
\newcommand\ste[1]{{#1}_+}
\newcommand\out{^{\operatorname{out}}}
\newcommand\lad{\mathcal{L}}
\newcommand\hsl{\widehat{\mathfrak{sl}}}
\newcommand\fkh{\mathfrak{h}}
\newcommand\GG{H}
\newcommand\lset[2]{\left\{\left.#1\,\right|\,#2\right\}}
\newcommand\rset[2]{\left\{#1\,\left|\,#2\right.\right\}}
\newcommand\bset[1]{\mathcal{Q}(#1)}
\newcommand\kpo{k\hspace{-0.02em}{+}\hspace{-0.06em}1}
\newcommand\kpt{k\hspace{-0.02em}{+}\hspace{-0.06em}2}
\newcommand\dpo{d\hspace{-0.02em}{+}\hspace{-0.06em}1}
\newcommand\apo{a\hspace{-0.02em}{+}\hspace{-0.06em}1}
\newcommand\ipo{i\hspace{-0.02em}{+}\hspace{-0.06em}1}
\newcommand\apt{a\hspace{-0.02em}{+}\hspace{-0.06em}2}
\newcommand\ipt{i\hspace{-0.02em}{+}\hspace{-0.06em}2}
\newcommand\aph{a\hspace{-0.02em}{+}\hspace{-0.06em}3}
\newcommand\bmo{b\hspace{-0.02em}{-}\hspace{-0.06em}1}
\newcommand\bmt{b\hspace{-0.02em}{-}\hspace{-0.06em}2}
\newcommand\jmt{j\hspace{-0.02em}{-}\hspace{-0.06em}2}
\newcommand\jmo{j\hspace{-0.02em}{-}\hspace{-0.06em}1}
\newcommand\imo{i\hspace{-0.02em}{-}\hspace{-0.06em}1}
\newcommand\gmo{g\hspace{-0.02em}{-}\hspace{-0.06em}1}

\newcommand\yo[1]{\raising{\small{$\begin{tikzpicture}[scale=0.42]#1\end{tikzpicture}$}}}
\newcommand\yw[1]{\raising{\footnotesize{$\begin{tikzpicture}[xscale=1.3,scale=0.47]#1\end{tikzpicture}$}}}

\newcommand{\h}{\mathcal{H}}
\newcommand{\sym}{\mathfrak{S}} 
\newcommand{\C}{\mathbb{C}}
\newcommand{\au}{a^\ast(\la)}
\newcommand{\auc}{a^\ast(\la')}
\newcommand{\ad}{a_\ast(\la)}
\newcommand{\adc}{a_\ast(\la')}
\newcommand{\bu}{b^\ast(\la)}
\newcommand{\buc}{b^\ast(\la')}
\newcommand{\bd}{b(\la)}
\newcommand{\bdc}{b(\la')}
\newcommand{\mone}{^-\hspace{-1pt}1}
\newcommand{\mtwo}{^-\hspace{-1pt}2}
\newlength\raiser
\newcommand\raising[1]{\settoheight\raiser{#1}\setlength\raiser{0.5\raiser}\addtolength\raiser{-.5pt}\raisebox{-\raiser}{\raisebox{3pt}{#1}}}

\newcommand\bx[3]{\draw[fill=white](#1,#2)--++(0,1)--++(1,0)--++(0,-1)--cycle;\draw(#1+0.5,#2+0.5)node[fill=white,inner sep=0.5pt]{$\vphantom1\smash{#3}$};}
\newcommand\blx[2]{\fill[white](#1-0.2,#2-0.2)rectangle++(1.4,1.4);}
\newcommand\hdb[3]{\draw(#1,#2)--++(#3,0);\draw(#1,#2+1)--++(#3,0);\draw[thick,dotted](#1+0.2*#3,#2+0.5)--++(0.6*#3,0);}
\newcommand\vdb[3]{\draw(#1,#2)--++(0,#3);\draw(#1+1,#2)--++(0,#3);\draw[thick,dotted](#1+0.5,#2+0.2*#3)--++(0,0.6*#3);}
\newcommand\hda[3]{\draw[dashed](#1,#2)--++(#3,0);}
\newcommand\vda[3]{\draw[dashed](#1,#2)--++(0,#3);}

\newcommand\hdbxy[1]{\draw(\xpos,\ypos)--++(#1,0);\draw(\xpos,\ypos+1cm)--++(#1,0);\draw[thick,dotted](\xpos+0.2*#1cm,\ypos+0.5cm)--++(0.6*#1,0);}
\newcommand\hdbxyr[1]{\hdbxy{#1}\addtolength\xpos{#1cm}}
\newcommand\vdbxy[1]{\draw(\xpos,\ypos)--++(0,#1);\draw(\xpos+1cm,\ypos)--++(0,#1);\draw[thick,dotted](\xpos+0.5cm,\ypos+0.2*#1cm)--++(0,0.6*#1);}
\newcommand\vdbxyu[1]{\vdbxy{#1}\addtolength\ypos{#1cm}}

\newlength\xpos\newlength\ypos\setlength\xpos{0cm}\setlength\ypos{0cm}
\newcommand\setxy{\setlength{\xpos}{0cm}\setlength{\ypos}{0cm}}
\newcommand\bxy[1]{\draw[fill=white](\xpos,\ypos)--++(0,1)--++(1,0)--++(0,-1)--cycle;\draw(\xpos+0.5cm,\ypos+0.5cm)node[fill=white,inner sep=0.5pt]{$\vphantom1\smash{#1}$};}
\newcommand\bxu[1]{\bxy{#1}\addtolength{\ypos}{1cm}}
\newcommand\bxd[1]{\bxy{#1}\addtolength{\ypos}{-1cm}}
\newcommand\bxr[1]{\bxy{#1}\addtolength{\xpos}{1cm}}
\newcommand\bxl[1]{\bxy{#1}\addtolength{\xpos}{-1cm}}

\title[Reducible Specht modules]{The reducible Specht modules for the Hecke algebra $\h_{\mathbb{C},{-}1}(\sym_n)$}

\author[M.~Fayers]{Matthew Fayers}
\address{Queen Mary University of London, Mile End Road, London E1 4NS, UK}
\email{m.fayers@qmul.ac.uk}
\author[S.~Lyle]{Sin\'ead Lyle}
\address{School of Mathematics \\ University of East Anglia \\ Norwich NR4 7TJ \\ UK}
\email{s.lyle@uea.ac.uk}

\begin{abstract}
The reducible Specht modules for the Hecke algebra $\h_{\bbf,q}(\sym_n)$ have been classified except when $q={-}1$.  We prove one half of a conjecture which we believe classifies the reducible Specht modules when $q={-}1$ and $\bbf$ has characteristic 0.
\end{abstract}
\maketitle

\maketitle

\section{Introduction} \label{S:Intro}
Fix a field $\bbf$ of characteristic $p \gs 0$ and an element $q \in \bbf^\times$.  For $n\gs 0$, the Hecke algebra $\h_n=\h_{\bbf,q}(\sym_n)$ of the symmetric group $\sym_n$ is defined to be the unital associative $\bbf$-algebra with generators $T_1,\ldots,T_{n-1}$ subject to the relations
\begin{align*}
(T_i-q)(T_i+1)& =0 & &\text{for } 1 \ls i \ls n-1, \\
T_jT_{j+1}T_j& =T_{j+1}T_jT_{j+1} &&\text{for } 1 \ls j \ls n-2, \\
T_iT_j &=T_jT_i &&\text{for } 1 \ls i<j-1\ls n-2.
\end{align*}
Note that if $q=1$ then $\h_n \cong \bbf\sym_n$.  For each partition $\la$ of $n$, Dipper and James defined an $\h_n$-module $S^\la$ known as a Specht module.  
An important open problem in representation theory is to determine the decomposition matrices of the Hecke algebras; this is equivalent to determining the composition factors of the Specht module $S^\la$ for each partition $\la$.  An interesting special case of this problem is the question of which Specht modules are irreducible.  For the symmetric group algebra $\bbf\sym_n$, the answer to this question is completely known, and for the Hecke algebra $\h_n$, the answer is known except in the case where $q={-}1$~\cite{JM2,L1,F1,F2,JLM,L2,F3}.  In this paper we prove one half of a conjecture (Conjecture~\ref{C:Main} below) that describes the irreducible Specht modules when $q={-}1$ and $p=0$, and we give a conjecture for the case of positive characteristic.

The layout of the paper is as follows.  In Section~\ref{S:Main} we give some background on partitions and Specht modules, and state the main result of this paper, Theorem~\ref{T:Main}.  In Section~\ref{S:Proof} we describe some results and techniques for proving reducibility of Specht modules, and use these to prove Theorem~\ref{T:Main} subject to the proof of Proposition~\ref{P:MH}; this is a technical result on homomorphisms, which requires a long proof.  In Section~\ref{S:Hom} we give detailed background on homomorphisms between Specht modules and prove Proposition~\ref{P:MH}.

\section{The main theorem} \label{S:Main}
Throughout Section~\ref{S:Main}, we assume that $q={-}1$ and that $\bbf$ has characteristic $p \gs 0$. Recall that a \emph{composition} of $n$ is a sequence $\la=(\la_1,\la_2,\ldots)$ of non-negative integers such that $\sum_{i=1}^\infty \la_i =n$.  If in addition  $\la_1 \gs \la_2 \gs \cdots$, we say that $\la$ is a \emph{partition} of $n$.  When writing a partition, we usually omit zeroes, and group together equal positive parts with a superscript.  We let $\ell(\la)$ denote the number of non-zero parts of $\la$, and we write $|\la|$ to mean $\sum_{i=1}^\infty\la_i$.

The \emph{Young diagram} of $\la$ is the set
\[\rset{(r,c)}{1 \ls c \ls \la_r} \subset \mathbb{N}^2,\]
whose elements we call the \emph{nodes} of $\la$. Throughout this paper, we identify $\la$ with its Young diagram; so for example we may write $\la\subseteq\mu$ to mean that $\la_i\ls\mu_i$ for all $i$.  We use the English convention for drawing Young diagrams, in which the first coordinate increases down the page and the second increases from left to right.

A node $\mathfrak{u} \in \la$  is said to be \emph{removable} if $\la \setminus \mathfrak{u}$ is a partition and a node $\mathfrak{v} \notin \la$  is said to be \emph{addable} if $\la \cup \mathfrak{v}$ is a partition.
The \emph{$2$-residue} of a node $(r,c) \in \mathbb{N}^2$, which we shall simply call the residue, is defined to be $(c-r) \ppmod 2$.  
The partition $\la$ is said to be \emph{$2$-regular} if $\la_i>\la_{i+1}$ for all $1 \ls i <\ell(\la)$ and is said to be \emph{$2$-restricted} if $\la_i-\la_{i+1}\ls 1$ for all $i\gs1$.  If $\la$ is not $2$-regular, we will say it is \emph{$2$-singular}.

If $\la$ is a partition, we write $S^\la$ for the \emph{Specht module}, as defined by Dipper and James~\cite{DJ}.  If $\la$ is $2$-regular then $S^\la$ has a unique irreducible quotient $D^\la$, and the set $\{D^\la \mid \la \text{ is $2$-regular}\}$ is a complete set of non-isomorphic irreducible $\h_n$-modules.  The \emph{conjugate} $\la'$ of a partition $\la$ is defined to be the partition whose Young diagram is given by $\{(c,r) \mid 1 \ls c \ls \la_r\}$. Conjugation is useful in this paper because of the following result.

\begin{lemma}\thmcite{FL}{Corollary~3.3}\label{L:Conj}
Suppose $\la$ is a partition of $n$.  Then $S^\la$ is irreducible if and only if $S^{\la'}$ is irreducible.  
\end{lemma}

\subs{Irreducible Specht modules in characteristic zero}

We now discuss the problem of classifying irreducible Specht modules.  In this section we assume that $\bbf$ has characteristic zero.

The classification of irreducible Specht modules labelled by $2$-regular partitions is well known.  In characteristic zero this takes the following simple form.

\begin{propn}\thmcite{JM2}{Theorem~4.15} \label{P:Carter} 
Let $\la$ be a partition of $n$ and suppose that $\la$ is $2$-regular.  Then $S^\la$ is irreducible if and only if $\la_i - \la_{i+1}$ is odd for all $1 \ls i < \ell(\la)$. 
\end{propn}

We say that $\la$ is \emph{alternating} if it satisfies the condition of Proposition \ref{P:Carter}. Using this proposition and Lemma \ref{L:Conj}, it remains only  to classify the irreducible Specht modules $S^\la$ when $\la$ and $\la'$ are both $2$-singular; we call such a partition \emph{doubly-singular}.  A conjecture for this classification has been given by the first author and Mathas.  First we need to make a definition.

\begin{defn}
Let $\la$ be a doubly-singular partition of $n$. Set
\begin{itemize}
\item $a$ to be maximal such that $\la_a-\la_{a+1} \gs 2$,
\item $b$ to be maximal such that $\la_b=\la_{b+1}\gs 1$, and
\item $c$ to be maximal such that $\la_{a+c}>0$.
\end{itemize}
Say that $\la$ is an \emph{FM-partition} if the following conditions all hold.
\begin{itemize}
\item $\la_i-\la_{i+1} \ls 1$ for all $i \neq a$.
\item $\la_b \gs a-1 \gs b$.
\item $\la_1>\dots>\la_c$.
\item If $c=0$ then all addable nodes of $\la$ except possibly those in the first row and first column have the same residue.
\item If $c>0$, then all addable nodes of $\la$ have the same residue. 
\end{itemize}
\end{defn}

\begin{conj} \label{C:Main}
Let $\h_n=\h_{\bbf,{-}1}(\sym_n)$ where $\cha(\bbf)=0$, and let $\la$ be a doubly-singular partition of $n$.  The $\h_n$-module $S^\la$ is irreducible if and only if $\la$ or $\la'$ is an FM-partition. 
\end{conj}

The main result of this paper is the proof of half of this conjecture.

\vspace{\topsep}

\noindent\hspace{-3pt}\fbox{\ \parbox{\textwidth}{\vspace{-\topsep}
\begin{thm} \label{T:Main}
Let $\h_n=\h_{\bbf,{-}1}(\sym_n)$ where $\cha(\bbf)=0$, and let $\la$ be a doubly-singular partition of $n$.  If the $\h_n$-module $S^\la$ is irreducible then $\la$ or $\la'$ is an FM-partition. 
\end{thm}
\vspace{-\topsep}}\ }

\Needspace*{5em}
\subs{Irreducible Specht modules in positive characteristic}

We make some brief comments on the case where $\bbf$ has prime characteristic $p$ (and $q={-}1$).  In this case, the classification of irreducible Specht modules remains unsolved, but here we conjecture a solution.

For the case of Specht modules labelled by $2$-regular or $2$-restricted partitions, a more complicated version of Proposition \ref{P:Carter} (also covered by \cite[Theorem~4.15]{JM2}) holds, so the difficulty lies with doubly-singular partitions.  In this case, the theory of decomposition maps shows that Theorem \ref{T:Main} still holds; however, there are FM-partitions which label reducible Specht modules in positive characteristic.

Recall that if $\la$ is a partition and $(r,c)$ is a node of $\la$, then the \emph{$(r,c)$-hook length} of $\la$ is the integer
\[
h_{r,c}(\la)=\la_r-r+\la'_c-c+1.
\]
Given a positive integer $s$, we say that $\la$ is an \emph{$s$-core} if none of the hook lengths of $\la$ is divisible by $s$.  Now we have the following result, proved by the first author in \cite{F3}.

\begin{thm}\label{2pcore}
Suppose $\bbf$ has characteristic $p$ and $q={-}1$.  If $\la$ is a doubly-singular partition which is not a $2p$-core, then the $\h_n$-module $S^\la$ is reducible.
\end{thm}

The results of \cite{F3} also show that for a given prime $p$ there are only finitely many FM-partitions which are also $2p$-cores.  So in order to complete the classification of irreducible Specht modules in a given non-zero characteristic, there are only finitely many Specht modules to consider.  Based on computer calculations, we now make the following conjecture.

\begin{conj}
Let $\h_n=\h_{\bbf,{-}1}(\sym_n)$ where $\cha(\bbf)=p>0$, and let $\la$ be a doubly-singular partition of $n$.  The $\h_n$-module $S^\la$ is irreducible if and only if $\la$ is a $2p$-core and $\la$ or $\la'$ is an FM-partition. 
\end{conj}

The case $p=2$ of this conjecture (which amounts to the classification of irreducible Specht modules for the symmetric group in characteristic $2$) is the main result of \cite{JM3}.  Using computer programs written in GAP \cite{gap}, the first author has been able to verify the conjecture also for $p=3,5$ and $7$.

\section{The proof of Theorem~\ref{T:Main}} \label{S:Proof}
Throughout Section~\ref{S:Proof}, we assume that $q={-}1$ and that $\bbf$ is a field of characteristic 0.  Our aim is to prove Theorem~\ref{T:Main}, that is, if $\la$ is a doubly-singular partition of $n$ and the $\h_{\bbf,{-}1}(\sym_n)$-module $S^\la$ is irreducible then $\la$ or $\la'$ is an FM-partition.  

\subs{Techniques for proving reducibility}\label{S:Techniques}

We begin by describing some methods -- some well-known and some new -- which can be used to prove the reducibility of a Specht module.

\subsubs{Ladders}

For $k \gs 1$, the $k$th \emph{ladder} in $\bbn^2$ is defined to be the set of nodes
\[\mathcal{L}_k = \rset{(i,j) \in \mathbb{N}^2}{i+j=k+1}.\]
We say that $\call_l$ is a \emph{longer} ladder than $\call_k$ if $l>k$.

The $k$th ladder of a partition $\la$ is the intersection of $\call_k$ with the Young diagram of $\la$.  We say that the $k$th ladder of $\la$ is \emph{broken} if the nodes it contains are not consecutive in $\call_k$; that is, there exist $1\ls r<s<t\ls k$ such that $(r,k+1-r)$ and $(t,k+1-t)$ lie in $[\la]$ but $(s,k+1-s)$ does not.

The following proposition is the main result of~\cite{FL}.

\begin{propn}\thmcite{FL}{Theorem~2.1}\label{P:BrokeRed}
Suppose that $\la$ has a broken ladder.  Then $S^\la$ is reducible.  
\end{propn}

A more helpful description of the condition in Proposition \ref{P:BrokeRed} is as follows: $\la$ has a broken ladder if there exist $1 \ls a<b$ such that $\la_{a}- \la_{a+1} \gs 2$ and $\la_b=\la_{b+1}>0$.

\subsubs{Regularisation and homomorphisms}

Recall that the \emph{dominance order} $\dom$ on partitions is defined by saying that $\mu \dom \la$ if and only if
\[\sum_{i=1}^l \mu_i \gs \sum_{i=1}^l \la_i\tag*{for all $l \gs 1$.}\]

If $\la$ is a partition, let $\la^\reg$ denote the partition whose Young diagram is obtained by moving the nodes of $\la$ as high as possible in their ladders.  It is easy to see that $\la^\reg$ is a $2$-regular partition, and that $\la^\reg\dom\la$.  We also have $\la^\reg=(\la')^\reg$ for any $\la$.

For example, if $\la=(3,2^3)$, then $\la^\reg=(5,3,1)$; this can be seen from the following diagrams, in which we label the nodes of these two partitions with the numbers of the ladders in which they appear.
\[
\yo{\bxd3\bxd4\bxl5\bxu4\bxu3\bxu2\bxr1\bxr2\bxr3\addtolength\xpos{5cm}\bxr1\bxr2\bxr3\bxr4\bxd5\addtolength\xpos{-2cm}\bxl4\bxl3\bxd2\bxr3}
\]
The importance of regularisation lies in the following result.

\begin{lemma}\thmcite{J}{Theorem~6.21} \label{L:Reg}
Let $\la$ be a partition of $n$.  Then $D^{\la^\reg}$ occurs as a composition factor of $S^\la$ with multiplicity $1$. If $D^\nu$ is a composition factor of $S^\la$ then $\nu \dom \la^\reg$.  
\end{lemma}

This result is particularly useful when classifying irreducible Specht modules, since it implies that if $S^\la$ is irreducible, then $S^\la\cong D^{\la^\reg}$.  One application of this is as follows.

\begin{cory} \label{C:Hom}
Suppose $\la$ and $\mu$ are partitions of $n$, such that $\la^\reg \ndom \mu$ and $\hom_{\h_n}(S^\mu,S^\la) \neq0$.  Then $S^\la$ is reducible.
\end{cory}

\begin{proof}
Since $\hom_{\h_n}(S^\mu,S^\la) \neq 0$, the $\h_n$-modules $S^\mu$ and $S^\la$ have a common composition factor, $D^\nu$ say.  By Lemma~\ref{L:Reg} we have $\nu \dom \mu^\reg \dom \mu$, so $\nu \neq \la^\reg$.  So $S^\la$ has at least two composition factors.    
\end{proof}

We shall apply Corollary \ref{C:Hom} using two different explicit constructions of homomorphisms.  The first is a $q$-analogue, due to the second author, of the `one-node homomorphisms' constructed by Carter and Payne in \cite{cp}.

\begin{defn} 
Say that a partition $\la$ is \emph{CP-reducible} if $\la$ has
\begin{itemize} 
\item an addable node lying in ladder $\mathcal{L}_m$, and
\item a removable node lying in ladder $\mathcal{L}_l$,
\end{itemize}
where $m >l$ and $l \equiv m \ppmod 2$.
\end{defn}
\cite[Theorem 4.1.1]{L2} shows that if $\la$ is a CP-reducible partition of $n$, then there is a partition $\mu$ of $n$ with $\mu\ndom\la^\reg$, and a non-zero $\h_n$-homomorphism from $S^\mu$ to either $S^\la$ or $S^{\la'}$.  Hence by Lemma \ref{L:Conj} and Corollary \ref{C:Hom} we have the following.

\begin{propn}\thmcite{FL}{Proposition~4.6} \label{P:CPRed}
Suppose that $\la$ is CP-reducible.  Then $S^\la$ is reducible.   
\end{propn}

Now we give the second result we require on homomorphisms.  This also defines a certain family of pairs of partitions where the corresponding homomorphism space is non-zero; however, the partitions in question are rather less natural than in the Carter--Payne case, and the result below was proved solely for the purposes of the present paper.

\begin{defn}
Say that a partition $\la$ is \emph{MH-reducible} if there exists $x\gs0$ such that $(x+1,\la_{x+1}+1)$ is an addable node of $\la$, and the partition $\nu=(\la_{x+1},\la_{x+2},\dots)$ has the form
\[\left((g+f+s')^s,g+f+s'-1,g+f+s'-2,\ldots,g+s',g,g-1,\ldots,2\right)\]
where $s,s',f,g$ are integers such that $f\gs0$, $g\gs2$, $s'\gs s\gs2$ and either 
\begin{itemize}
\item $s$ and $s'$ are odd; or
\item $s=2$, $s'$ is even and $f=0$.  
\end{itemize}
\end{defn}

\begin{propn}\label{P:MH}
Suppose $\la$ is a partition of $n$ which is MH-reducible, and let $x$ be as in the definition of MH-reducible.  Define a partition $\mu$ by
\[\mu_i = \begin{cases}
\la_i+1 & (i=x+1,x+2) \\
\la_i-2&(i=\ell(\la))\\
\la_i & (\text{otherwise}). 
\end{cases}\]
Then $\hom_{\h_n}(S^\mu,S^\la) \neq0$.  
\end{propn}

The partitions appearing in Proposition \ref{P:MH} may be visualised using the following diagram (in which we take $s=s'=5$, $g=4$, $f=2$).  The dotted nodes at the bottom of the diagram are present in $\la$, while those at the top right are present in $\mu$.

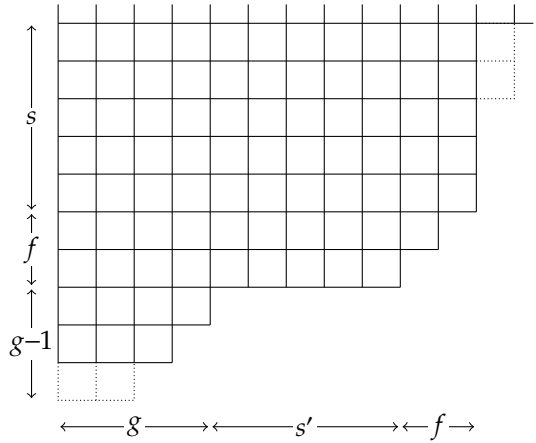
\newlength\inm
\setlength\inm{0.06cm}
{\small\[
\begin{tikzpicture}[scale=0.5]
\foreach \x in {0,1,2,3,4,5,6,7,8,9,10,11}
\draw(\x,5)--(\x,10.5);
\foreach \y in {5,6,7,8,9}
\draw(0,\y)--(11,\y);
\draw(0,10)--(12.5,10);
\draw(12,10)--(12,10.5);
\foreach \x in {0,1,2,3}
\draw(\x,1)--(\x,5);
\draw(0,1)--(3,1);
\draw(0,2)--(4,2);
\draw(0,3)--(9,3);
\draw(0,4)--(10,4);
\draw(4,2)--(4,5);
\foreach \x in {5,6,7,8,9}
\draw(\x,3)--(\x,5);
\draw(10,4)--(10,5);
\draw[densely dotted](0,1)--(0,0)--(2,0)--(2,1);
\draw[densely dotted](1,0)--(1,1);
\draw[densely dotted](11,8)--(12,8)--(12,10)--(11,10);
\draw[densely dotted](11,9)--(12,9);
\draw[<->](-0.7,0cm+\inm)--(-0.7,3cm-\inm);
\draw[<->](-0.7,3cm+\inm)--(-0.7,5cm-\inm);
\draw[<->](-0.7,5cm+\inm)--(-0.7,10cm-\inm);
\draw[<->](0cm+\inm,-0.7)--(4cm-\inm,-0.7);
\draw[<->](4cm+\inm,-0.7)--(9cm-\inm,-0.7);
\draw[<->](9cm+\inm,-0.7)--(11cm-\inm,-0.7);
\draw(-0.7,1.5)node[fill=white,inner sep=1pt]{$\gmo$};
\draw(-0.7,4)node[fill=white,inner sep=1pt]{$f$};
\draw(-0.7,7.5)node[fill=white,inner sep=1pt]{$s$};
\draw(2,-0.7)node[fill=white,inner sep=1pt]{$g$};
\draw(6.5,-0.7)node[fill=white,inner sep=1pt]{$s'$};
\draw(10,-0.7)node[fill=white,inner sep=1pt]{$f$};
\end{tikzpicture}
\]}

The proof of Proposition~\ref{P:MH} is somewhat lengthy, and we postpone it to Section~\ref{S:Hom}, where we introduce all the necessary background concerning homomorphisms.

\begin{propn} \label{P:MHRed}
Suppose that $\la$ is MH-reducible.  Then $S^\la$ is reducible.  
\end{propn}

\begin{proof}
Since $\la$ is MH-reducible, we may define the partition $\mu$ as in Proposition~\ref{P:MH} so that $\hom_{\h_n}(S^\mu,S^\la) \neq0$.  Furthermore, the condition $s'\gs s$ guarantees that $\mu$ is obtained from $\la$ by moving two nodes to longer ladders, so by \cite[Lemma~2.1]{F1}, $\la^\reg \ndom \mu$ and hence $S^\la$ is reducible by Corollary~\ref{C:Hom}.   
\end{proof}

\subsubs{Fock space techniques}

\begin{defn}
Say that a partition $\la$ with $\ell(\la)=l$ is \emph{LLT-reducible} if $\la$ is $2$-singular, has no broken ladders and satisfies:
\begin{itemize}
\item $\la_1 \gs l+1$;
\item $\la_l \gs 2$;
\item there exists $1 \ls x<l$ with $\la_x-\la_{x+1}>1$.
\end{itemize} 
\end{defn}

\begin{propn} \label{P:LLTRed}
Suppose $\la$ is LLT-reducible.  Then $S^\la$ is reducible.  
\end{propn}

Before proving Proposition~\ref{P:LLTRed} we give some background.  
In~\cite{FL}, the authors show how Ariki's Theorem~\cite{A} may be used to prove that certain Specht modules are reducible.  We summarise the relevant results here.  For details, and to put these results into context, we refer the reader to~\cite[Section~5]{FL}. 

Suppose that $\la$ is a partition. If $\mu$ is a partition such that $\mu \subseteq \la$ and $\mu_i-\mu_{i+1}$ is odd for $1 \ls i <\ell(\la)$, we will say that $\mu$ is \emph{alternating in $\la$}.  In this case, we define a sequence of partitions $\mu=\mu^0,\mu^1,\mu^2,\ldots$ by setting $\mu^{j+1}$ to be the partition obtained from $\mu^j$ by adding all addable nodes that are contained in $\la$.  Now define a $\la$-tableau $T=T(\la,\mu)$ as follows.  Begin by filling in each node of $\mu$ with a 0, then, for $j \gs 1$ fill in each node of $\mu^j \setminus \mu^{j-1}$ with $j$.  (Readers unfamiliar with tableaux should consult Section \ref{SS:ConHom} below.)

Now for each node $(r,c) \in \la$, let $j=T_{r,c}$ and define
\[
N(r,c) = \left|\rset{m<r}{T_{m,\la_m}<j,\ T_{m,\la_m} \not\equiv j \ppmod 2}\right| - \left|\rset{m<r}{T_{m,\la_m}<j,\ T_{m,\la_m} \equiv j \ppmod 2}\right|.
\]
Let $N=N(\la,\mu) = \sum_{(r,c) \in \la} N(r,c)$.  

\begin{eg}
Let $\la=(13,12^5,7,4,3)$ and $\mu=(13,12,11,10,9,8,7,2,1)$.  The tableaux $T,N$ are shown in the two diagrams below, and we see that $N(\la,\mu)=10$.
\[
T=
\yo{\bxl0\bxl0\bxl0\bxl0\bxl0\bxl0\bxl0\bxl0\bxl0\bxl0\bxl0\bxl0\bxd0
\bxr0\bxr0\bxr0\bxr0\bxr0\bxr0\bxr0\bxr0\bxr0\bxr0\bxr0\bxd0
\bxl1\bxl0\bxl0\bxl0\bxl0\bxl0\bxl0\bxl0\bxl0\bxl0\bxl0\bxd0
\bxr0\bxr0\bxr0\bxr0\bxr0\bxr0\bxr0\bxr0\bxr0\bxr0\bxr1\bxd2
\bxl3\bxl2\bxl1\bxl0\bxl0\bxl0\bxl0\bxl0\bxl0\bxl0\bxl0\bxd0
\bxr0\bxr0\bxr0\bxr0\bxr0\bxr0\bxr0\bxr0\bxr1\bxr2\bxr3\bxd4
\addtolength\xpos{-5cm}\bxl0\bxl0\bxl0\bxl0\bxl0\bxl0\bxd0
\bxr0\bxr0\bxr1\bxd2\addtolength\xpos{-1cm}\bxl2\bxl1\bxl0
}
\qquad\qquad 
N=
\yo{\bxl0\bxl0\bxl0\bxl0\bxl0\bxl0\bxl0\bxl0\bxl0\bxl0\bxl0\bxl0\bxd0
\bxr0\bxr0\bxr0\bxr0\bxr0\bxr0\bxr0\bxr0\bxr0\bxr0\bxr0\bxd0
\bxl2\bxl0\bxl0\bxl0\bxl0\bxl0\bxl0\bxl0\bxl0\bxl0\bxl0\bxd0
\bxr0\bxr0\bxr0\bxr0\bxr0\bxr0\bxr0\bxr0\bxr0\bxr0\bxr2\bxd\mone
\bxl2\bxl\mone\bxl2\bxl0\bxl0\bxl0\bxl0\bxl0\bxl0\bxl0\bxl0\bxd0
\bxr0\bxr0\bxr0\bxr0\bxr0\bxr0\bxr0\bxr0\bxr2\bxr\mone\bxr2\bxd\mone
\addtolength\xpos{-5cm}\bxl0\bxl0\bxl0\bxl0\bxl0\bxl0\bxd0
\bxr0\bxr0\bxr3\bxd\mtwo
\addtolength\xpos{-1cm}\bxl\mtwo\bxl3\bxl0
}
\]
\end{eg}

The following lemma follows from \cite[Lemma 5.4 \& Lemma 5.5]{FL}.

\begin{lemma} \label{L:LLT}
Let $\la$ be a partition of $n$.  Suppose that $\mu$ and $\tilde\mu$ are alternating in $\la$.  If $N(\la,\mu) \neq N(\la,\tilde\mu)$ then $S^\la$ is reducible.
\end{lemma}

We can now prove Proposition~\ref{P:LLTRed}.

\begin{proof}[Proof of Proposition~\ref{P:LLTRed}]
Suppose that $\la$ is LLT-reducible, let $l=\ell(\la)$ and let $x<l$ be maximal such that $\la_{x} -\la_{x+1}>1$; since $\la$ has no broken ladders, we have $\la_i-\la_{i+1}=1$ for $x+1 \ls i <l$.  Now we consider two cases.
\begin{itemize}
\item
Suppose first that $\la_1+1 \gs \la_l+l$.  Define $\sigma$ as follows.  Set $\sigma_1$ to be maximal such that $\sigma_1 \ls \la_1$ and $\sigma_1+1 \equiv \la_l+l \ppmod 2$.  For $2 \ls i \ls x$, define $\sigma_i$ to be maximal such that $\sigma_i \ls \min\{\la_i, \sigma_{i-1}-1\}$ and $\sigma_i+i \equiv \la_l+l \ppmod 2$.  Since $\la$ has no broken ladders and $\la_1+1\gs\la_l+l$, we have $\la_i+i \gs \la_l+l$ for all $1 \ls i \ls x$, and using this it is easy to show by induction that $\sigma_i\gs \la_l+l-i$ for all $1\ls i\ls x$.  In particular, $\sigma_x \gs \la_l+l-x >\la_l+l-x-1=\la_{x+1}$.  So we can define two partitions $\mu$ and $\tilde\mu$ by setting
\begin{align*}
\mu&=(\sigma_1,\sigma_2,\ldots,\sigma_x,\la_{x+1},\la_{x+2},\ldots,\la_{l}), \\
\tilde\mu&=(\sigma_1,\sigma_2,\ldots,\sigma_x,\la_{x+1}-2,\la_{x+2}-2,\ldots,\la_{l}-2).
\end{align*} 
By construction, $\mu$ and $\tilde\mu$ are alternating in $\la$, and we claim that $N(\la,\mu)\neq N(\la,\tilde\mu)$.  The entries in $T(\la,\mu)$ and $T(\la,\tilde\mu)$ agree except in the last two entries in rows $x+1,\dots,l$, which are $\yo{\bxr0\bxy0}$ in $T(\la,\mu)$, and $\yo{\bxr1\bxy2}$ in $T(\la,\tilde\mu)$.  So the definition of $N(\la,\mu)$ gives
\[N(\la,\tilde\mu) - N(\la,\mu) = (l-x) \left|\rset{1 \ls m \ls x}{T_{m,\la_m} =1}\right|.\]
Choose $1 \ls g \ls x$ minimal such that $\sigma_g \neq \la_g$.  Then by construction $\la_g - \sigma_g=1$ and $T_{g,\la_g}=1$.  Hence $N(\la,\tilde\mu) - N(\la,\mu)>0$ and $S^\la$ is reducible by Lemma \ref{L:LLT}.
\item
Now suppose that $\la_1+1 < \la_l+l$.  Define $\mu$ and $\tilde\mu$ by
\begin{align*}
\mu&=(\la_1,\la_1-1,\ldots,\la_1-x+1,\la_{1}-x,\ldots,\la_{1}-l+1), \\
\tilde\mu&= (\la_1,\la_1-1,\ldots,\la_1-x+1,\la_{1}-x-2,\ldots,\la_{1}-l-1).
\end{align*} 
Again we claim that $\mu$ and $\tilde\mu$ satisfy the conditions of Lemma~\ref{L:LLT}.  Since $\la_1+1<\la_l+l$, the nodes $(1,\la_1)$, $(l,\la_1-l+1)$ and $(l,\la_1-l)$ all lie in $\la$, so since $\la$ has no broken ladders, $\mu$ and $\tilde\mu$ are both alternating in $\la$.  Again, $T(\la,\mu)$ and $T(\la,\tilde\mu)$ agree except in rows $x+1,\dots,l$.  If we let $k=\la_l-\la_1+l-1$, then rows $x+1,\dots,l$ of $T(\la,\mu)$ have the form
$\yo{\bxr0\hdbxyr{1.5}\bxr0\bxr1\bxr2\hdbxyr{1.5}\bxy k}$,
while in $T(\la,\tilde\mu)$ these rows have the form
$\yo{\bxr0\hdbxyr{1.5}\bxr0\bxr1\bxr2\hdbxyr{1.5}\draw[fill=white](\xpos,\ypos)--++(0,1)--++(1.5,0)--++(0,-1)--cycle;\draw(\xpos+0.75cm,\ypos+0.5cm)node[fill=white,inner sep=0.5pt]{$\vphantom1\smash{\kpt}$};}$.
Hence
\[N(\la,\tilde\mu) - N(\la,\mu) = (l-x) \left|\rset{1 \ls m \ls x}{T_{m,\la_m}=k+1}\right|.\]
It remains to show that $T_{m,\la_m}=k+1$ for some $1 \ls m \ls x$, which is equivalent to saying that the ladder $\mathcal{L}=\mathcal{L}_{l+\la_l}$  intersects non-trivially with the set of nodes $\{(m,\la_m) \mid 1 \ls m \ls x\}$.      Certainly $\mathcal{L}$ intersects with $\{(x,c) \mid 1 \ls c \ls \la_x\}$ since $\la_x-\la_{x+1} >1$, so choose $r \gs 1$ minimal such that $\mathcal{L}$ intersects with $\{(r,c) \mid 1 \ls c \ls \la_r\}$.  Then $r>1$ since $\la_1+1 < \la_l+l$, so the fact that $r$ is minimal means that $(r,\la_r)$ lies on $\mathcal{L}$ as required.\qedhere
\end{itemize}
\end{proof}

\subsubs{Induction and restriction}

\begin{defn}
For $i \in \{0,1\}$ let $\la^{(i)}$ be the partition obtained by removing all removable nodes of residue $i$ from $\la$.  
\end{defn}

The proof of the following proposition comes from (the $q$-analogue of)~\cite[Lemma~2.13]{BK}. 

\begin{propn}\thmcite{FL}{Lemma~3.13}\label{P:IndRed}
Suppose $i\in\{0,1\}$. If $S^{\la^{(i)}}$ is reducible then so is $S^\la$.
\end{propn}

Obviously this result will enable us to prove Theorem \ref{T:Main} by induction.  In order to do this, we make the following definition.

\begin{defn}
Say that $\la$ is \emph{inductively reducible} if for some $i\in \{0,1\}$ we have $\la^{(i)} \neq \la$ and one of the following holds.
\begin{itemize}
\item $\la^{(i)}$ is $2$-regular or $2$-restricted and $S^{\la^{(i)}}$ is reducible.
\item $\la^{(i)}$ is doubly-singular and neither $\la^{(i)}$ nor $\la^{(i)'}$ is an FM-partition.  
\end{itemize}     
\end{defn}

\subs{Analysis of partitions}\label{S:Classify}

The aim of this section is to complete the proof of Theorem \ref{T:Main}, modulo the proof of Proposition \ref{P:MH}.  The strategy is simple: we show that a Specht module which is not shown to be reducible by any of the techniques in \S\ref{S:Techniques} is labelled by an FM-partition or the conjugate of one.  That is, we prove the following result.

\begin{propn} \label{P:Classify} Suppose $\la$ is a partition of $n$ which satisfies the following conditions:
\begin{itemize}
\item $\la$ is doubly-singular;
\item $\la$ does not have a broken ladder;
\item neither $\la$ nor $\la'$ is CP-reducible;
\item $\la$ is not MH-reducible;
\item neither $\la$ nor $\la'$ is LLT-reducible;
\item $\la$ is not inductively reducible.
\end{itemize}
Then either $\la$ or $\la'$ is an FM-partition.  
\end{propn}

Throughout this section we fix a partition $\la$ with $\ell(\la)=l$ satisfying the hypotheses of Proposition \ref{P:Classify}.  We begin by introducing some additional notation.

\begin{defn} Suppose $\nu$ is a partition.
If $\nu$ is not $2$-restricted, we define:
\begin{itemize}
\item $a^{\ast}(\nu)$ to be minimal such that $\nu_{a^\ast(\nu)}-\nu_{a^\ast(\nu)+1} \gs 2$;
\item $a_{\ast}(\nu)$ to be maximal such that $\nu_{a_\ast(\nu)}-\nu_{a_\ast(\nu)+1} \gs 2$;
\item $c(\nu)$ to be maximal such that $\nu_{a_\ast(\nu)+c(\nu)}>0$.
\end{itemize}
If $a^\ast(\nu)=a_\ast(\nu)$ we will write $a(\nu)=a^\ast(\nu)=a_\ast(\nu)$.

If $\nu$ is not $2$-regular, we define $b(\nu)$ to be maximal such that $\nu_{b(\nu)}=\nu_{b(\nu)+1}>0$.
\end{defn}

Now consider our chosen partition $\la$. Since $\la$ has no broken ladders, we have $\ad \gs \au > \bd$ and $\auc> \bdc$.

\begin{defn}
Say that $\la$ is \emph{pointed} if $b(\la)+1 = a^{\ast}(\la)$.  Note that $\la$ is pointed if and only if $\la'$ is pointed.  If $\la$ is pointed, we call the removable node in row $a^{\ast}(\la)$ the \emph{point}.  
\end{defn}

\begin{lemma} \label{PossibleResidues}
If $\la$ is not pointed then all removable nodes of $\la$ have the same residue.  If $\la$ is pointed then all removable nodes of $\la$ except possibly the point have the same residue.  
\end{lemma}

\begin{proof}
First note that any removable node $(r,c)$ of $\la$ which is not the point has an addable node adjacent to it: if $r<a^\ast(\la)$, then the node $(r+1,c)$ must be addable, while if $r>b(\la)+1$ then the node $(r,c+1)$ is addable.

Now suppose there are two removable nodes $(r,c)$ and $(r',c')$ of different residues, neither of which is the point, lying in ladders $k,k'$ say.  Since the residues of the nodes are not the same we have $k\neq k'$ and we suppose without loss of generality that $k<k'$.  There is an addable node adjacent to $(r',c')$, and this must lie in ladder $k'+1$.  Since $k'+1>k$ and $k'+1\equiv k\ppmod2$, $\la$ is CP-reducible; contradiction.
\end{proof}

\begin{cory}\label{C:Alladd}
All addable nodes of $\la$ except possibly those in the first row and first column have the same residue.  
\end{cory}

\begin{proof}
Every addable node except possibly those in the first row or column has a removable node (which is not the point) adjacent to it.
\end{proof}

\begin{defn}
Define $\mu$ to be the partition obtained from $\la$ by removing all removable nodes if all the removable nodes have the same residue, and all removable nodes except the point otherwise. 
\end{defn}

Note that $\mu\neq\la$, so since $\la$ is not inductively reducible, either $\mu$ or $\mu'$ must be either alternating or an FM-partition.

The following properties of $\mu$ follow easily from the definitions.

\begin{lemma}  \label{RestIneq} \quad\vspace{-\topsep}
\begin{itemize}
\item Suppose that $\mu$ is not $2$-restricted.  Then $a^\ast(\la)=a^{\ast}(\mu)$.  
\item Suppose that $\mu$ is not $2$-regular.  Then $\la_{b(\la)}=\mu_{b(\mu)}$.  
\item Suppose that $\la_l \neq 2$ or that $\mu_l=\la_l$.  Then $\mu$ is not $2$-restricted and $a_\ast(\la)=a_\ast(\mu)$.  
\item Suppose that $\la_1>\la_2$ or that $\la_2=\la_3$ or that $\mu_2=\la_2$.  Then $\mu$ is not $2$-regular.
\end{itemize}
\end{lemma}

\begin{lemma} \label{UpFrommu}
Suppose that all addable nodes of $\la$ have the same residue and that $\la_l \neq 2$ or $\la_l=\mu_l$.  If $\mu$ is an FM-partition then $\la$ is an FM-partition.
\end{lemma}

\begin{proof}
By Lemma~\ref{RestIneq} we have
\[\au =a^\ast(\mu) =a_\ast(\mu)=\ad,\]
\[\la_{\bd} = \la_{b(\mu)} \gs a(\mu)-1 = a(\la)-1.\]
It remains only to show that $\la_1 >\dots >\la_{c(\la)}$.  If $c(\la) \ls 1$ there is nothing to check, so assume that $c(\la)\gs 2$.  Then $c(\mu)=c(\la)-1 \gs 1$, so all the addable nodes of $\mu$ have the same residue.  This means that the node $(1,\la_1+1)$ cannot be an addable node of $\mu$, so $\la_1>\la_2$.  Now since we have $\mu_1>\dots>\mu_{c(\la)-1}$ and $\la$ does not have a broken ladder, we must have $\la_1>\dots>\la_{c(\la)}$.
\end{proof}

\begin{lemma} \label{LastPartLarge}
Suppose that $\la_l \gs 3$ and $\mu$ is an FM-partition.  Then $\la$ is an FM-partition.
\end{lemma}

\begin{proof}
Suppose that $\mu$ is an FM-partition.  Then using Lemma~\ref{RestIneq} 
\begin{align*}
\au&=a^\ast(\mu)=a_{\ast}(\mu)=\ad,\\
\la_{\bd}&= \mu_{b(\mu)} \gs a(\mu)-1 = a(\la)-1,\\
c(\la)&=c(\mu) =0. 
\end{align*}    
Since $a(\la)-1 \gs \bd$ and all addable nodes of $\la$, except for possibly those in the first row and the first column, have the same residue, $\la$ is also an FM-partition.
\end{proof}

\begin{lemma} \label{InEq}
Suppose $\la_1> l$ and $\la_l \gs 2$.  Then $\la$ is an FM-partition.
\end{lemma}

\begin{proof}
Since $\la$ is not LLT-reducible we have $a^\ast(\la)=a_\ast(\la)=l$.  Since $\la_1>l$ and $\la$ is not $2$-regular this implies that $\la_l \gs 3$ and therefore $a^\ast(\mu)=a_{\ast}(\mu)=l$ and $c(\mu)=0$.  If $\bd=1$ then $\la=((l+1)^2,l,l-1,\dots,3)$ is an FM-partition, so assume $\bd>1$.  Then $\mu$ is doubly-singular, so either $\mu$ or $\mu'$ is an FM-partition.  In fact, we claim that $\mu$ must be an FM-partition.  If $\mu'$ is an FM-partition then we have $c(\mu')\ls1$ (because $\mu'_1=\mu'_2=l$) and $a^\ast(\mu') = a_\ast(\mu')$, so that
\[\mu_{b(\mu)}=\mu_1-c(\mu')\gs\la_1-2\gs l-1 =a(\mu)-1.\] 
Hence $\mu_{b(\mu)} \gs a(\mu)-1$ and $\mu$ is also an FM-partition.

Now Lemma~\ref{LastPartLarge} implies that $\la$ is also an FM-partition.
\end{proof}

\begin{lemma} \label{l1}
Suppose that $\la_1 \gs l$ and $\la_l=1$.  Then $\la$ or $\la'$ is an FM-partition.
\end{lemma}

\begin{proof}
Since $\la_1\gs l$, the addable node $(1,\la_1+1)$ lies in a longer ladder than the removable node $(l,1)$.  Since $\la$ is not CP-reducible, these nodes must have different residues, so the addable node $(1,\la_1+1)$ has the same residue as the addable nodes $(l,2)$ and $(l+1,1)$.  So by Corollary \ref{C:Alladd} all the addable nodes of $\la$ have the same residue.

Now we claim that $\mu$ is doubly-singular.  By Lemma \ref{RestIneq} $\mu$ is not $2$-restricted, and the only way $\mu$ could be $2$-regular is if $\la_1=\la_2>\la_3$.  But if this is the case then the removable nodes $(2,\la_1)$ and $(l,1)$ of $\la$ have different residues, so $(2,\la_1)$ must be the point; and this means that $\mu_1=\mu_2$, so $\mu$ is not $2$-regular.

So either $\mu$ or $\mu'$ is an FM-partition.  If $\mu$ is an FM-partition, then by Lemma \ref{UpFrommu} $\la$ is an FM-partition.  If $\mu'$ is an FM-partition, then (from the argument in the last paragraph) either $\la'_{\ell(\la')}\neq2$ or $\mu'_{\ell(\la')}=\la'_{\ell(\la')}$; so by Lemma \ref{UpFrommu} $\la'$ is an FM-partition.
\end{proof}

\begin{lemma} \label{l2}
Suppose that $\la_1=\la_2=l$ and $\la_l=2$.  If $\mu$ is an FM-partition then $\la$ or $\la'$ is an FM-partition.  
\end{lemma}

\begin{proof}
First note that since $\la_1=\la_2$ and $\la_l=2$, we cannot have $\au=\ad$, because this would give $\la=(l^2,l-1,l-2,\dots,2)$ so that $\mu$ is $2$-regular.  So $\au<\ad$; since $\mu$ is an FM-partition, we have $a^\ast(\mu)=a_\ast(\mu)=a^\ast(\la)$.

Let $x \gs 0$ be minimal such that $\la_{x+1}=\la_{\bd}$ and let $\nu=(\la_{x+1},\ldots,\la_l)$.  Since $a^\ast(\mu)=a_\ast(\mu)$, $\nu$ has the form 
\[\nu=\left((g+f+s')^s,g+f+s'-1,\ldots,g+s',g,\ldots,2\right)\]
where $g \gs 2$, $f \gs 0$ and $s,s' \gs 2$.  If $\la = \nu$ then $\la_1=\la_2$ so, since $\mu$ is an FM-partition, $c(\mu)=1$; that is, $g=2$.  Then $\la'_{\bdc}= l-1 = \auc -1$ and $\la'$ is an FM-partition.  Assume then that $\la \neq \nu$.  Then we have
\[
x+s+f+g-1=l,\qquad g+s'+f=\la_{x+1}\gs \la_1-x+1=l-x+1,
\]
which gives $s' \gs s$.

Suppose all removable nodes of $\la$ have the same residue.  Then $s$ and $s'$ are both odd, so $\la$ is MH-reducible, a contradiction.

Next suppose that not all removable nodes of $\la$ have the same residue.  Then $f=0$ and $s,s'$ are even.  Let $\sigma$ be the partition obtained by removing the point of $\la$; then $\sigma$ is doubly-singular, so either $\sigma$ or $\sigma'$ is an FM-partition.  In particular, either $a^\ast(\sigma)=a_\ast(\sigma)$ or $a^\ast(\sigma')=a_\ast(\sigma')$, which means that either $s$ or $s'$ equals $2$.  Since $s'\gs s$, we get $s=2$, so again $\la$ is MH-reducible; contradiction.  
\end{proof}

By combining the results in this section, we can prove Proposition \ref{P:Classify}.

\begin{proof}[Proof of Proposition \ref{P:Classify}]
Suppose $\la$ is a partition with the given properties.  Then $\la'$ has the same properties, and we may replace $\la$ with $\la'$ if necessary.

If $\la_1>l$, then by Lemma \ref{InEq} or Lemma \ref{l1}, either $\la$ or $\la'$ is an FM-partition.  So we may assume $\la_1\ls l$.  Applying the same argument to $\la'$, we may assume that $\la_1=l$.

Now by Lemma \ref{l1} applied to both $\la$ and $\la'$, we can assume that $\la_l\gs2$ and $\la_1=\la_2$.  Now the only way $\mu$ could be $2$-regular or $2$-restricted is if $\la=(l^2,l-1,l-2,\dots,2)$, which is an FM-partition.  So we can assume that $\mu$ is doubly-singular.  Hence either $\mu$ or $\mu'$ is an FM-partition.  Replacing $\la$ with $\la'$ if necessary, we can assume $\mu$ is an FM-partition.  And now we are done using Lemma \ref{LastPartLarge} or Lemma \ref{l2}.
\end{proof}

We are now ready to prove Theorem~\ref{T:Main}.  

\begin{proof}[Proof of Theorem \ref{T:Main}]
The proof is by induction on $|\la|$.  If $\la=\emptyset$ the theorem is trivially true.  Suppose that $\la$ is a doubly-singular partition of $n\gs 1$ such that neither $\la$ nor $\la'$ is an FM-partition, and suppose that Theorem~\ref{T:Main} holds for all partitions of $m<n$.  By Proposition~\ref{P:Classify} at least one of the following statements holds for $\la$.
\begin{itemize}
\item $\la$ has a broken ladder.
\item $\la$ or $\la'$ is CP-reducible.
\item $\la$ is MH-reducible.  
\item $\la$ or $\la'$ is LLT reducible.
\item $\la$ is inductively reducible.
\end{itemize}
If any of the first four statements hold then $S^\la$ is reducible by Lemma~\ref{L:Conj},  Proposition~\ref{P:BrokeRed}, Proposition~\ref{P:CPRed}, Proposition~\ref{P:MHRed} or Proposition~\ref{P:LLTRed}.  So suppose $\la$ is inductively reducible.  Then there exists $i\in \{0,1\}$ such that $\la^{(i)} \neq \la$ and $\la^{(i)}$ satisfies one of the following conditions.
\begin{itemize}
\item $\la^{(i)}$ is $2$-regular and is not alternating.
\item $\la^{(i)'}$ is $2$-regular and is not alternating.
\item $\la^{(i)}$ is doubly-singular and neither $\la^{(i)}$ nor $\la^{(i)'}$ is an FM-partition.  
\end{itemize}     
By Proposition~\ref{P:Carter} or the inductive hypothesis, $S^{\la^{(i)}}$ is reducible.  Then $S^\la$ is reducible by Proposition~\ref{P:IndRed}.   
\end{proof}

To complete the proof of Theorem~\ref{T:Main} it remains only to give the deferred proof of Proposition~\ref{P:MH}.  

\section{Homomorphisms between Specht modules} \label{S:Hom}

\subs{Constructing homomorphisms}\label{SS:ConHom}

\subsubs{Tableaux}

If $\mu$ is a composition of $n$, a \emph{$\mu$-tableau} is defined to be a filling of the nodes of $\mu$ with positive integers; if $T$ is a tableau, we write $T_{r,c}$ for the $(r,c)$-entry.  The \emph{type} of a tableau is the composition $\la$, where $\la_i$ is the number of nodes filled with the integer $i$, for each $i$.  A tableau is \emph{row-standard} if the entries are weakly increasing along the rows. We write $\calt(\mu,\la)$ for the set of row-standard $\mu$-tableaux of type $\la$.  If $\mu$ is a partition, we say that a $\mu$-tableau is \emph{semistandard} if the entries are weakly increasing along the rows and strictly increasing down the columns; we write $\calt_0(\mu,\la)$ for the set of semistandard $\mu$-tableaux of type $\la$.  We remark that $\calt_0(\mu,\la)$ is empty unless $\mu\dom{\arrowover\la}$, where $\arrowover\la$ is the partition obtained by arranging the parts of $\la$ in decreasing order.

\subsubs{Permutation modules and Specht modules}

Now take $\bbf$ to be an arbitrary field with $q \in \bbf^\times$.  For each composition $\la$ of $n$, we let $M^\la$ denote the `permutation module' defined by Dipper and James; if $\la$ is a partition, then the Specht module $S^\la$ is a submodule of $M^\la$.  If $\mu,\la$ are compositions of $n$ and $T$ is a row-standard $\mu$-tableau of type $\la$, then there is an $\h_n$-homomorphism $\check\Theta_T:M^\mu\to M^\la$. The set $\lset{\check\Theta_T}{T\in\calt(\mu,\la)}$ is a basis for $\hom_{\h_n}(M^\mu,M^\la)$ \cite[Theorem 3.4]{DJ}.

These homomorphisms may be used to define the Specht module.  Suppose $\la$ and $\mu$ are partitions of $n$, and $1 \ls d < \ell(\la)$ and $1 \ls t \ls \la_{d+1}$. Define the partition $\la(d,t)$ by
\[\la(d,t)_i = \begin{cases} \la_i+t & (i=d) \\
\la_i-t & (i=d+1) \\
\la_i & (\text{otherwise}).
\end{cases} \]
Then there is a unique row-standard $\la$-tableau of type $\la(d,t)$ with the property that for every $i\neq d+1$ all the entries in row $i$ are equal to $i$.  The corresponding homomorphism from $M^\la$ to $M^{\la(d,t)}$ is denoted $\psi_{d,t}$.

The Kernel Intersection Theorem~\cite[Theorem~7.5]{DJ} says that
\[
S^\la = \bigcap_{d=1}^{\ell(\la)-1} \bigcap_{t=1}^{\la_{d+1}} \ker(\psi_{d,t}).
\]

\begin{rmk}
Our notation is not universally used: the partition $\la(d,t)$ is referred to elsewhere in the literature as $\nu(d,t)$; we use the notation $\la(d,t)$ in order to emphasise the dependence on $\la$.  In addition, the homomorphism $\psi_{d,t}$ is sometimes denoted $\psi_d^t$ or $\psi_{d,\la_{d+1}-t}$.
\end{rmk}

If $\mu$ is a partition and $\la$ a composition of $n$ and $T\in\calt(\mu,\la)$, we shall often consider the restriction of $\check\Theta_T$ to $S^\mu$, which we denote $\Theta_T$.  We write $\ehom_{\h_n}(S^\mu,M^\la)$ for the subspace of $\hom_{\h_n}(S^\mu,M^\la)$ spanned by all the $\Theta_T$; by \cite[Corollary 8.7]{dj2}, $\lset{\Theta_T}{T\in\calt_0(\mu,\la)}$ is a basis for $\ehom_{\h_n}(S^\mu,M^\la)$; in particular, $\ehom_{\h_n}(S^\mu,M^\la)=0$ unless $\mu\dom{\arrowover\la}$.

\begin{rmk}
In fact, $\ehom_{\h_n}(S^\mu,M^\la)$ is almost always equal to $\hom_{\h_n}(S^\mu,M^\la)$; the exception is the case of most interest in this paper, when $q={-}1$ and $\mu$ is $2$-singular.

We also remark that homomorphisms denoted $\check\Theta_T,\Theta_T$ are denoted $\Theta_T,\hat\Theta_T$ elsewhere in the literature.  Since we shall almost exclusively be considering the restricted homomorphism, we use the less cluttered notation for this.
\end{rmk}

\subsubs{Constructing homomorphisms between Specht modules}\label{homspecht}

Suppose now that $\la,\mu$ are partitions of $n$, and $\Theta\in\hom_{\h_n}(S^\mu,M^\la)$.  By the Kernel Intersection Theorem, we have $\im(\Theta) \subseteq S^\la$ if and only if $\psi_{d,t}\circ\Theta =0$ for all $d,t$.  We shall only be considering the cases where $\Theta\in\ehom_{\h_n}(S^\mu,M^\la)$; we write $\ehom_{\h_n}(S^\mu,S^\la)$ for the set of $\Theta\in\ehom_{\h_n}(S^\mu,M^\la)$ for which $\im(\Theta)\subseteq S^\la$.

It turns out that it is possible to give an expression for $\psi_{d,t}\circ\Theta_T$, which shows in particular that $\psi_{d,t}\circ\Theta_T\in\ehom_{\h_n}(S^\mu,M^{\la(d,t)})$.  One consequence of this which will save a lot of effort later is that we automatically have $\psi_{d,t}\circ\Theta_T=0$ unless $\mu\dom{\arrowover{\la(d,t)}}$.

In order to give our expression for $\psi_{d,t}\circ\Theta_T$, we need to recall quantum integers and quantum binomial coefficients. For $m \gs 0$ define
\[[m] = 1+q+\dots + q^{m-1},\]
and $[m]!=\prod_{i=1}^m[i]$.  If $q$ is an indeterminate, then for integers $m,j$, set
\[
\gauss{m}{j} = 
\begin{cases}
\mfrac{[m]!}{[j]![m-j]!} & (m\gs j\gs0)\\
0&(\text{otherwise}).
\end{cases}
\]
Then $\gauss{m}{k}$ is a polynomial in $q$; so we can extend the definition of $\gauss mk$ to the case where $q$ is algebraic by defining it to be the specialisation of this polynomial.

For a tableau $T$, let $T^i_j$ denote the number of entries equal to $i$ in row $j$ of $T$.  Let $T^{>i}_j = \sum_{k>i}T^k_j$, and we define terms such as $T^{<i}_j$ similarly.  Now we can describe the composition $\psi_{d,t}\circ\Theta_T$.

\begin{propn}\thmcite{L2}{Proposition~2.14}\label{Lemma5}  
Suppose that $T$ is a row-standard $\mu$-tableau of type $\la$.   
Choose $d$ with $1 \ls d <\ell(\la)$ and $t$ with $1 \ls t \ls \la_{d+1}$.  Let $\mathcal{S}$ be the set of row-standard tableaux of type $\la(d,t)$ obtained by replacing $t$ of the entries in $T$ which are equal to $d+1$ with $d$.
Then 
\[\psi_{d,t}\circ\Theta_T = \sum_{S \in \mathcal{S}} \left(  \prod_{j= 1}^{\ell(\mu)} q^{T^d_{>j}(S^d_j - T^d_j)} \gauss{S^d_j}{T^d_j}\right) \Theta_S.\] 
\end{propn}

A difficulty with Proposition \ref{Lemma5} is that it expresses $\psi_{d,t}\circ\Theta_T$ in terms of homomorphisms labelled by tableaux which are not necessarily semistandard.  In order to be able to use this result to show that a composition $\psi_{d,t}\circ\Theta$ is zero, we need the following result, which allows a homomorphism $\Theta_T$ to be written in terms of other tableaux. In this proposition, we write $\bbz_+$ for the set of non-negative integers; given $g\in\bbz_+^l$, we write $\bar g_{d-1}$ for the partial sum $\sum_{i=1}^{d-1}g_i$.

\begin{propn}\thmcite{L3}{Theorem~4.2}\label{Lemma7}
Suppose $\mu$ is a partition and $\nu$ a composition of $n$, and $S\in\calt(\mu,\nu)$.
\begin{enumerate}
\item Suppose $1 \ls r\ls \ell(\mu)-1$ and that $1 \ls d \ls \ell(\nu)$.  Let
\[\mathcal{G} =\rset{g\in\bbz_+^{\ell(\nu)}}{g_d=0, \, \textstyle\sum_{i=1}^{\ell(\nu)} g_i =S^d_{r+1} \text{ and } g_i \ls S^{i}_{r} \text{ for } 1 \ls i \ls \ell(\nu)}.\]
For $g \in \mathcal{G}$, let
$U_g$ be the row-standard tableau formed from $S$ by moving all entries equal to $d$ from row $r+1$ to row $r$ and for $i \neq d$ moving $g_i$ entries equal to $i$ from row $r$ to row $r+1$. Then
\[\Theta_S = ({-}1)^{S^d_{r+1}} q^{-\binom{S^d_{r+1}+1}{2}} q^{-S^d_{r+1}S^{<d}_{r+1}} \sum_{g \in \mathcal{G}} q^{\bar{g}_{d-1}} \prod_{i=1}^{\ell(\nu)} q^{g_i S^{<i}_{r+1}} \gauss{S^i_{r+1}+g_i}{g_i}\Theta_{U_g}.\]
\item Suppose $1 \ls r\ls \ell(\mu)-1$ and $\mu_r=\mu_{r+1}$ and that $1 \ls d \ls \ell(\nu)$.  Let
\[\mathcal{G} =\rset{g\in\bbz_+^{\ell(\nu)}}{g_d=0, \, \textstyle\sum_{i=1}^{\ell(\nu)} g_i = S^d_r \text{ and } g_i \ls S^{i}_{r+1} \text{ for } 1 \ls i \ls \ell(\nu)}.\]
For $g \in \mathcal{G}$, let 
$U_g$ be the row-standard tableau formed form $S$ by moving all entries equal to $d$ from row $r$ to row $r+1$ and for $i \neq d$ moving $g_i$ entries equal to $i$ from row $r+1$ to row $r$. Then
\[\Theta_S =  ({-}1)^{S^d_{r}} q^{-\binom{S^d_{r}}{2}} q^{-S^d_r S^{>d}_r} \sum_{g \in \mathcal{G}} q^{-\bar{g}_{d-1}}  \prod_{i=1}^{\ell(\nu)} q^{g_i S^{>i}_{r}} \gauss{S^i_{r}+g_i}{g_i} \Theta_{U_g}.\]
\end{enumerate}
\end{propn}

We remark that since the first draft of this paper was written, the first author has proved a more general result giving linear relations between tableau homomorphisms \cite{f4}, which yields an explicit fast algorithm for `semistandardising' a homomorphism.  However, the result above will be sufficient in this paper.

The following result~\cite[Theorem 3.1]{LM} or~\cite[Prop.~10.4]{D} often allows us to simplify our calculations.  

\begin{propn}\label{P:RowRem}
Suppose that $\la$ and $\mu$ are partitions of $n$ and that for some $x \gs 0$ we have $\la_i=\mu_i$ for $1 \ls i \ls x$.  Let $\bar\la=(\la_{x+1},\la_{x+2},\ldots)$ and $\bar\mu=(\mu_{x+1},\mu_{x+2},\ldots)$, and let $m=|\bar\la|=|\bar\mu|$.  Then
\[\dim_\bbf\ehom_{\h_n}(S^\mu,S^\la) = \dim_\bbf\ehom_{\h_m}(S^{\bar\mu},S^{\bar\la}).\] 
\end{propn}

We will also make use of the next result.  

\begin{lemma}\label{toomany}
Suppose $V$ is a $\la$-tableau such that for some $k$ there are $m$ entries equal to $k$ which all lie in rows of length strictly less than $m$.  Then $\Theta_V=0$.
\end{lemma}

\begin{proof}
Choose $y$ minimal such that $V^k_y \neq 0$.  
We may apply Proposition~\ref{Lemma7} repeatedly to write $\Theta_V$ as a linear combination of homomorphisms indexed by tableaux obtained by moving all entries equal to $k$ in $V$ upwards until they are all contained in row $y$.  But by assumption there are no such tableaux.   
\end{proof}

We may now use Proposition~\ref{Lemma5} and Proposition~\ref{Lemma7} to give a proof of Proposition~\ref{P:MH}.  We now return to the assumption that $\bbf$ has characteristic 0 and that $q={-}1$.  
By Proposition~\ref{P:RowRem} the proof of Proposition~\ref{P:MH} follows from the following proposition.

\begin{propn} \label{mainhom}
Fix integers $s,s',f,g$ with $f\gs0$, $g\gs2$, $s'\gs s\gs2$ and 
\begin{itemize}
\item $s$ and $s'$ are odd; or
\item $s=2$, $s'$ is even and $f=0$.  
\end{itemize}
Define
\begin{align*}
\mu &= \left((g+f+s'+1)^2,(g+f+s')^{s-2},g+f+s'-1,g+f+s'-2,\dots,g+s',g,g-1,\dots,3\right),\\
\la &= \left((g+f+s')^{s},g+f+s'-1,g+f+s'-2,\dots,g+s',g,g-1,\dots,2\right),
\end{align*}
and let $n=|\la|=|\mu|$. Then 
\[\ehom_{\h_n}(S^\mu,S^\la)\neq0.\]
\end{propn}

The remainder of this paper is devoted to proving Proposition~\ref{mainhom}.  

\subs{Proof of Proposition~\ref{mainhom} when $s$ and $s'$ are both odd} \label{S:MHRedOdd}

Fix integers $s,s',f,g,\la,\mu$ as in the statement of Proposition~\ref{mainhom} and assume $s$ and $s'$ are odd.  Let $l=\ell(\mu)=s+f+g-2$.  We will say a $\mu$-tableau (of arbitrary type) is \emph{usable} if for every row $i$, all except possibly the last two entries are equal to $i$.  All the tableaux we consider will be usable.  Given a usable tableau of shape $\mu$, we will often encode it simply by giving a tableau of shape $(2^l)$, recording the last two entries in each row.  Conversely, given a tableau of shape $(2^l)$, we will talk about the corresponding usable $\mu$-tableau.

Now we need some more definitions.  Suppose $1\ls i<j\ls s$.  Then there is a unique $(2^{s-1})$-tableau $S(i,j)$ of type $(1^2,2^{s-2})$ such that
\begin{align*}
S(i,j)_{1,2}& =i, \quad S(i,j)_{2,2}=j,\\ 
S(i,j)_{1,1}&\ls S(i,j)_{2,1}\ls S(i,j)_{3,1}\ls S(i,j)_{3,2}\ls S(i,j)_{4,1}\ls\cdots\ls S(i,j)_{s-1,2}.\tag*{($\ast$)}
\end{align*}
Define
\[
m_{ij}=\begin{cases}
\frac12(s-1)&(i=2,\ j=3)\\
({-}1)^{j+1}&(\text{otherwise}).
\end{cases}
\]
Later we shall also need a slight variant of the above definition.  Suppose $1\ls d\ls s-1$ and let
\[
\nu^d=\begin{cases}
(2,0,2^{s-2})&(d=1)\\
(1,2,1,2^{s-3})&(d=2)\\
(1^2,2^{d-3},3,1,2^{s-d-1})&(d\gs3);
\end{cases}
\]
that is, $\nu^d$ is the composition obtained from $(1^2,2^{s-2})$ by increasing the $d$th part by $1$ and decreasing the $(d+1)$th part by $1$.  Given $1 \ls i < j \ls s$ as above, but excluding the cases where $d=1$ and $i$ or $j$ is equal to 2, there is a unique $(2^{s-1})$-tableau $S^d(i,j)$ of type $\nu^d$ satisfying ($\ast$).

Next, we need to consider tableaux shape $(2^{g-1})$ and type $(2^{g-1})$.  Given such a tableau $T$ and given $1\ls i \ls g-1$, we will say that $T$ is \emph{split} at row $i$ if all the entries in rows $1,\dots,i$ are less than all the entries in rows $i+1,\dots,g-1$.  Let $\cala$ denote the set of $(2^{g-1})$tableaux $T$ of type $(2^{g-1})$ for which:
\begin{itemize}
\item
the entries in each row are weakly increasing;
\item
for each $k$, the entries in row $k$ are at least $k-1$;
\item
for all $2\ls k\ls g-2$, the first entry in row $k$ is strictly less than the second entry in row $k+1$;
\item
if $T$ is split at row $k$, then it is split at all rows $k+1,k+2,\dots,g-1$.
\end{itemize}
If $T \in \cala$, we define $\sgn(T)$ to be $({-}1)^a$, where $a$ is the first row at which $T$ is split.  

Now we can construct the semistandard $\mu$-tableaux which we will combine to give our homomorphism.  Set 
\[\mathcal{I}=\rset{(i,j)}{1 \leq i<j\leq s \text{ and $j$ is odd or } i \geq 3}.\]
Given $(i,j) \in \mathcal{I}$ and $T \in \cala$, construct a tableau of shape $(2^{s+f+g-2})$ as follows:
\begin{itemize}
\item
the first $s-1$ rows are just the rows of $S(i,j)$;
\item
for $s\ls k\ls s+f-1$, the entries in row $k$ are both equal to $k+1$;
\item
rows $s+f,\dots,s+f+g-2$ are the rows of $T$, with each entry increased by $s+f$.
\end{itemize}
Let $U(i,j,T)$ be the corresponding usable $\mu$-tableau, and let $\Theta(i,j,T)$ denote the corresponding homomorphism from $S^\mu$ to $M^\la$.

\begin{eg}
Suppose $(s,s',f,g)=(5,5,2,5)$.  Then $(2,5)\in\mathcal{I}$, and the tableau
\[
T=
\yo{\bxr2\bxd3\bxl1\bxd1\bxr2\bxd3\bxl4\bxl4}
\]
lies in $\cala$.  We have $m_{2,5}=1$ and $\sgn(T)={-}1$, and
\[
S(2,5)=
\yo{\bxr1\bxd2\bxl5\bxd3\bxr3\bxd4\bxl5\bxl4},\qquad
U(2,5,T)=
\yo{
\bxl2\bxl1\bxl1\bxl1\bxl1\bxl1\bxl1\bxl1\bxl1\bxl1\bxl1\bxl1\bxd1
\bxr2\bxr2\bxr2\bxr2\bxr2\bxr2\bxr2\bxr2\bxr2\bxr2\bxr2\bxr3\bxd5\addtolength\xpos{-1cm}
\bxl4\bxl3\bxl3\bxl3\bxl3\bxl3\bxl3\bxl3\bxl3\bxl3\bxl3\bxd3
\bxr4\bxr4\bxr4\bxr4\bxr4\bxr4\bxr4\bxr4\bxr4\bxr4\bxr4\bxd5
\bxl6\bxl6\bxl5\bxl5\bxl5\bxl5\bxl5\bxl5\bxl5\bxl5\bxl5\bxd5
\bxr6\bxr6\bxr6\bxr6\bxr6\bxr6\bxr6\bxr6\bxr6\bxr7\bxd7\addtolength\xpos{-1cm}
\bxl{10}\bxl9\bxl7\bxl7\bxl7\bxl7\bxl7\bxl7\bxl7\bxd7
\bxr8\bxr8\bxr8\bxr8\bxd8\addtolength\xpos{-1cm}
\bxl{10}\bxl9\bxl9\bxd9
\bxr{10}\bxr{11}\bxr{11}
}.
\]

\end{eg}

Then we claim that
\[
\Theta=\sum_{(i,j) \in \mathcal{I}}\sum_{T \in \cala} m_{i,j}\sgn(T)\Theta(i,j,T)
\]
gives a homomorphism in $\ehom_{\h_n}(S^\mu,S^\la)$.  One can check that all the $U(i,j,T)$ are semistandard, so $\Theta$ is certainly non-zero. All we then need to do is check that $\psi_{d,t}\circ\Theta =0$ for all $d,t$.  By dominance considerations (see the remarks in the second paragraph of \S\ref{homspecht}), the only pairs $d,t$ that we need to consider are $(d,1)$ for $1\ls d\ls s+f+g-2$, and $(d,2)$ for $s+1\ls d\ls s+f+g-2$. 

For later use, we extend the notation above: given $1\ls i<j\ls s$ and given $1\ls d\ls s-1$, we define $\Theta^d(i,j,T)$ in the same way, but using the tableau $S^d(i,j)$ instead of $S(i,j)$; as above, we exclude the cases where $d=1$ and $i$ or $j$ is equal to 2.  

\subsubs{Notation for tableaux}

We list here a few items of notation that we shall use below.
\begin{itemize}
\item
If $V,W$ are row-standard tableaux, we shall use the notation $V\sra dr W$ to mean that $W$ is obtained from $V$ by replacing a \pp{} in row $r$ with a \dd, and we write $V\sra d{r,s}W$ to mean that $W$ is obtained by replacing two \pp s with \dd s, in rows $r$ and $s$ (where $r$ may equal $s$).
\item
If $T$ is a tableau and $1\ls i\ls j$, we write $T\tub ij$ for the tableau consisting of rows $i,\dots,j$ of $T$.
\item
If $T,U$ are tableaux of the same shape and $1\ls i\ls j$, we write $T\agree ijU$ to mean that the entries of $T$ and $U$ are the same except in rows $i,\dots,j$.
\end{itemize}

\subsubs{Rows $1$ to $s$}

Throughout this section, we let $m=g+f+s'$.

\begin{propn}\label{topbit}
Suppose $(i,j) \in \mathcal{I}$, $T \in \cala$ and $1\ls d\ls s-1$.  Then $\psi_{d,1}\circ\Theta(i,j,T)$ is given by the following.
\renewcommand\theenumi{\alph{enumi}}
\renewcommand\labelenumi{(\theenumi)}
\begin{enumerate}
\item
$\psi_{d,1}\circ\Theta(i,d,T)=({-}1)^{m}\Theta^d(i,d,T)$\hfill if $d\gs3$ and $i<d$.
\item
$\psi_{d,1}\circ\Theta(i,d+1,T)=({-}1)^{m}\Theta^d(i,d,T)$\hfill if $d\gs 4$ and $i<d$.
\item
$\psi_{d,1}\circ\Theta(i,j,T)=0$\hfill if $d\gs 4$, $i<d$ and $j\neq d,d+1$.
\item
$\psi_{d,1}\circ\Theta(d,d+1,T)=({-}1)^{m+d+1}\Theta^d(2,d,T)$\hfill if $d\gs3$.
\item
$\psi_{d,1}\circ\Theta(d,j,T)=({-}1)^{m+1}\Theta^d(d,j,T)$\hfill if $d\gs2$ and $j\gs d+2$.
\item
$\psi_{d,1}\circ\Theta(d+1,j,T)=({-}1)^{m}\Theta^d(d,j,T)$\hfill if $d\gs2$ and $j\gs d+2$.
\item
$\psi_{d,1}\circ\Theta(i,j,T)=0$\hfill if $d\gs3$ and $i\gs d+2$.
\item
$\psi_{3,1}\circ\Theta(2,j,T)=({-}1)^{m+1}\Theta^3(2,3,T)$\hfill if $j\gs5$.
\item
$\psi_{2,1}\circ\Theta(2,3,T)=0$.
\item
$\psi_{2,1}\circ\Theta(i,j,T)=({-}1)^{m+i}\Theta^2(2,i,T)$\hfill if $i\gs4$.
\item
$\psi_{1,1}\circ\Theta(2,j,T)=0$.
\item
$\psi_{1,1}\circ\Theta(3,j,T)=({-}1)^{m}\Theta^1(1,j,T)$.
\item
$\psi_{1,1}\circ\Theta(i,j,T)=({-}1)^{m+i}\Theta^1(1,i,T)+({-}1)^{m+i+1}\Theta^1(1,j,T)$ \hfill if $i\gs4$.
\end{enumerate}
\renewcommand\theenumi{arabic{enumi}}
\renewcommand\labelenumi{\theenumi.}
\end{propn}

Given this, it is straightforward to check the following corollary.

\begin{cory}\label{topcory}
Suppose $T\in \cala$.  Then for $1\ls d\ls s-1$, we have
\[
\psi_{d,1}\circ\left(\sum_{(i,j) \in \mathcal{I}}m_{i,j}\Theta(i,j,T)\right)=0.
\]
Hence $\psi_{d,1}\circ\Theta =0$ for $1\ls d\ls s-1$.
\end{cory}

In order to prove Proposition \ref{topbit}, we need a few preliminary results concerning tableau homomorphisms.


\begin{lemma}\label{aaccbbb}
Suppose $\mu$ is a partition and $i\gs1$ is such that $\mu_{i+1}=\mu_i-1$.  Suppose $V,W,X$ are $\mu$-tableaux such that $V\agree i{i+1}W\agree i{i+1}X$ and
\[
V\tub i{i+1}=\yo{\setxy\bxr a\hdbxyr{2}\bxr a\bxr c\bxd c\setlength\xpos{0cm}\bxr b\hdbxyr{3}\bxy b},\qquad W\tub i{i+1}=\yo{\setxy\bxr a\hdbxyr{2}\bxr a\bxr b\bxd c\setlength\xpos{0cm}\bxr b\hdbxyr{2}\bxr b\bxy c},\qquad
X\tub i{i+1}=\yo{\setxy\bxr a\hdbxyr{2.5}\bxr a\bxr b\bxd b\setlength\xpos{0cm}\bxr b\hdbxyr{1.5}\bxr b\bxr c\bxy c},
\]
where $a<b<c$.  Then $\Theta_V=-\Theta_W-\Theta_X$.
\end{lemma}

\begin{proof}
Applying Proposition~\ref{Lemma7}, we get
\[
\Theta_V=({-}1)^{\binom m2+1}\Theta_Y+({-}1)^{\binom{m+1}2}\Theta_Z,
\]
where $m=\mu_{i+1}$, $Y\agree i{i+1}Z\agree i{i+1}V$ and
\[
Y\tub i{i+1}=\yo{\setxy\bxr b\hdbxyr3\bxr b\bxd c\setlength\xpos{0cm}\bxr a\hdbxyr2\bxr a\bxy c},\qquad
Z\tub i{i+1}=\yo{\setxy\bxr a\bxr b\hdbxyr{3.5}\bxd b\setlength\xpos{0cm}\bxr a\hdbxyr{1.5}\bxr a\bxr c\bxy c}.
\]
Proposition~\ref{Lemma7} again gives
\[
\Theta_Y=({-}1)^{\binom{m}2}\Theta_W,\qquad \Theta_Z=({-}1)^{\binom{m-1}2}\Theta_X.\tag*{\qedhere}
\]
\end{proof}

\begin{lemma}\label{aacdbbb}
Suppose $\mu$ is a partition and $i\gs1$ is such that $\mu_{i+1}=\mu_i-1$.  Suppose $V,W,X,Y$ are $\mu$-tableaux such that $V\agree i{i+1}W\agree i{i+1}X\agree i{i+1}Y$ and
\begin{alignat*}2
V\tub i{i+1}&=\yo{\setxy\bxr a\hdbxyr2\bxr a\bxr c\bxd d\setlength\xpos{0cm}\bxr b\hdbxyr3\bxy b},&\qquad
W\tub i{i+1}&=\yo{\setxy\bxr a\hdbxyr2\bxr a\bxr b\bxd d\setlength\xpos{0cm}\bxr b\hdbxyr2\bxr b\bxy c},\\
X\tub i{i+1}&=\yo{\setxy\bxr a\hdbxyr2\bxr a\bxr b\bxd c\setlength\xpos{0cm}\bxr b\hdbxyr2\bxr b\bxy d},&\qquad
Y\tub i{i+1}&=\yo{\setxy\bxr a\hdbxyr{2.5}\bxr a\bxr b\bxd b\setlength\xpos{0cm}\bxr b\hdbxyr{1.5}\bxr b\bxr c\bxy d},
\end{alignat*}
where $a<b<c<d$. Then $\Theta_V=-\Theta_W-\Theta_X-\Theta_Y$.
\end{lemma}

\begin{proof}
As in the proof of Lemma \ref{aaccbbb}, we apply Proposition~\ref{Lemma7} to $\Theta_V$ to move the \bb s up to row $1$, and then again to move the \ba s up to row $1$.
\end{proof}

\begin{lemma}\label{aabdbbbc}
Suppose $\mu$ is a partition with $\mu_i=\mu_{i+1}$ for some $i$, and $V,W$ are $\mu$-tableaux such that $V\agree i{i+1}W$ and
\[
V\tub i{i+1}=\yo{\setxy\bxr a\hdbxyr{1.5}\bxr a\bxr b\bxd d\setlength\xpos{0cm}\bxr b\hdbxyr{2.5}\bxr b\bxr c},\qquad
W\tub i{i+1}=\yo{\setxy\bxr a\hdbxyr{1.5}\bxr a\bxr b\bxd b\setlength\xpos{0cm}\bxr b\hdbxyr{1.5}\bxr b\bxr c\bxy d},
\]
where $a<b<c<d$. Then $\Theta_V=\Theta_W$.
\end{lemma}

\begin{proof}
By Proposition~\ref{Lemma7}(2), both homomorphisms equal $-\Theta_X$, where
\[
X\tub i{i+1}=\yo{\setxy\bxr a\hdbxyr{1.5}\bxr a\bxr c\bxd d\setlength\xpos{0cm}\bxr b\hdbxyr{3.5}\bxr b}.
\tag*{\qedhere}
\]
\end{proof}


\begin{lemma}\label{block}
Suppose $\mu$ is a partition and $1\ls a<b-1$ such that $\mu_a=\mu_{b-1}$, and that $X,Y$ are $\mu$-tableaux with $X\agree a{b-1}Y$ and
\[
X\tub a{b-1}=\yw{\setxy\bxr a\hdbxyr2\bxr a\bxd b
\setlength\xpos{0cm}\bxr\apo\hdbxyr2\bxr\apo\bxd\apt
\setlength\xpos{0cm}\bxr\apt\hdbxyr2\bxr\apt\bxd\aph
\setlength\xpos{0cm}\addtolength\ypos{-.5cm}\vdbxy{1.5}\addtolength\xpos{3cm}\vdbxy{1.5}\addtolength\xpos{1cm}\vdbxy{1.5}
\setlength\xpos{0cm}\addtolength\ypos{-1cm}\bxr\bmt\hdbxyr2\bxr\bmt\bxd\bmo\setlength\xpos{0cm}\bxr\bmo\hdbxyr2\bxr\bmo\bxr b},\qquad
Y\tub a{b-1}=\yw{\setxy\bxr a\hdbxyr{2.5}\bxr a\bxd \apo
\setlength\xpos{0cm}\bxr\apo\hdbxyr{1.5}\bxr\apo\bxr\apt\bxd\apt
\setlength\xpos{0cm}\bxr\apt\hdbxyr{1.5}\bxr\apt\bxr\aph\bxd\aph
\setlength\xpos{0cm}\addtolength\ypos{-.5cm}\vdbxy{1.5}\addtolength\xpos{2.5cm}\vdbxy{1.5}\addtolength\xpos{1cm}\vdbxy{1.5}\addtolength\xpos{1cm}\vdbxy{1.5}
\setlength\xpos{0cm}\addtolength\ypos{-1cm}\bxr\bmt\hdbxyr{1.5}\bxr\bmt\bxr\bmo\bxd\bmo
\setlength\xpos{0cm}\bxr\bmo\hdbxyr{1.5}\bxr\bmo\bxr b\bxr b}.
\]
Then $\Theta_X=-\Theta_Y$.
\end{lemma}

\begin{proof}
Define the tableau $Z$ by $Z\agree a{b-1}X$ and
\[Z\tub a{b-1}=\yw{
\setxy\bxr a\hdbxyr{2.5}\bxr a\bxd \apo
\setlength\xpos{0cm}\bxr\apo\hdbxyr{1.5}\bxr\apo\bxr b\bxd b
\setlength\xpos{0cm}\bxr\apt\hdbxyr{2.5}\bxr\apt\bxd\apt
\setlength\xpos{0cm}\addtolength\ypos{-.5cm}\vdbxy{1.5}\addtolength\xpos{3.5cm}\vdbxy{1.5}\addtolength\xpos{1cm}\vdbxy{1.5}
\setlength\xpos{0cm}\addtolength\ypos{-1cm}\bxr\bmt\hdbxyr{2.5}\bxr\bmt\bxd\bmt
\setlength\xpos{0cm}\bxr\bmo\hdbxyr{2.5}\bxr\bmo\bxr\bmo}.
\]
We define a sequence of tableaux $X=X_{b},X_{b-1},\ldots,X_{a+2}$, where for $k=b-1,\dots,a+2$, $X_{k}$ is formed from $X_{k+1}$ by swapping the $k$ in row $k-1$ and the $b$ in row $k$.  Applying Proposition~\ref{Lemma7}(2), we find that $\Theta_{X_k}=-\Theta_{X_{k-1}}$.  We then apply Proposition~\ref{Lemma7}(2) to $X_{a+2}$ to move the \bb{} from row $1$ to row $2$, so that $\Theta_{X_{a+2}}=-\Theta_Z$.  Hence $\Theta_X=({-}1)^{a+b+1}\Theta_Z$.

We do a similar thing for $Y$: for $k=b-1,\cdots,a+2$ we move the two $k$s from row $k-1$ to row $k$.  We get $\Theta_Y=({-}1)^{a+b}\Theta_Z$, which gives the result.
\end{proof}

Now we are ready to prove Proposition \ref{topbit}.

\needspace{0pt}
\begin{proof}[Proof of Proposition \ref{topbit}]
\renewcommand\theenumi{\alph{enumi}}
\renewcommand\labelenumi{(\theenumi)}
\indent
\begin{enumerate}
\vspace{-\topsep}
\item
This is a simple application of Proposition~\ref{Lemma5} and Proposition~\ref{Lemma7}, using the fact that $[m-1]-[m-2]=({-}1)^m$.
\item
$U(i,d+1,T)$ has \pp s in rows $2,d,d+1$, and Proposition~\ref{Lemma5} gives
\[
\psi_{d,1}\circ\Theta(i,d+1,T)=({-}1)^m\Theta^d(i,d,T)+[m]\Theta_V+\Theta_W,
\]
where $U(i,d+1,T)\sra dd V$ and $U(i,d+1,T)\sra d{d+1}W$.  Proposition~\ref{Lemma7} gives $\Theta_W=-[m-2]\Theta_V$, and the fact that $[m]=[m-2]$ gives the result.
\item
This is a simple application of Proposition~\ref{Lemma5} and Lemma \ref{toomany}.
\item
The \pp s in $U(d,d+1,T)$ lie in rows $2$, $d$ and $d+1$.  Proposition~\ref{Lemma5} gives
\[
\psi_{d,1}\circ\Theta(d,d+1,T)=({-}1)^{m+1}\Theta_V+[m]\Theta_W+\Theta_X,
\]
where
\[
U(d,d+1,T)\sra d2V,\qquad U(d,d+1,T)\sra ddW,\qquad U(d,d+1,T)\sra d{d+1}X.
\]
Proposition~\ref{Lemma7} gives $\Theta_X=-[m-2]\Theta_W$, and so we just need to show that $\Theta_V=({-}1)^d\Theta^d(2,d,T)$.  Applying Proposition~\ref{Lemma7}(2) rows $1$ and $2$ and then Lemma \ref{aaccbbb}, we find that $\Theta_V=\Theta_Y+\Theta_Z$, where
\[
Y\tub13=\yo{\setxy\bxr1\hdbxyr{3}\bxr1\bxd2\setlength\xpos{0cm}\bxr2\hdbxyr{2}\bxr2\bxr3\bxd d\setlength\xpos{0cm}\bxr3\hdbxyr{2}\bxr3\bxr d},\qquad
Z\tub13=\yo{\setxy\bxr1\hdbxyr{3.5}\bxr1\bxd2\setlength\xpos{0cm}\bxr2\hdbxyr{2.5}\bxr2\bxr3\bxd3\setlength\xpos{0cm}\bxr3\hdbxyr{1.5}\bxr3\bxr d\bxr d}
\]
and $Y\agree13Z\agree13U(d,d+1,T)$.  By Lemma \ref{toomany} we have $\Theta_Z=0$, so we concentrate on $\Theta_Y$.  For $k=4,\cdots,d-1$ the $k$th row of $Y$ consists entirely of $k$s.  So we can repeatedly apply Proposition~\ref{Lemma7}(2) to move the \dd{} in row $3$ down to row $d-1$, and we get $\Theta_Y=({-}1)^d\Theta^d(2,d,T)$, as required.
\item
If $d\gs3$, then this is a simple application of Proposition~\ref{Lemma5} and Proposition~\ref{Lemma7}: all the \pp s in $U(i,j,T)$ lie in rows $d$ and $d+1$, and Proposition~\ref{Lemma5} gives
\[
\psi_{d,1}\circ\Theta(d,j,T)=[m]\Theta^d(d,j,T)+\Theta_W,
\]
where $U(d,j,T)\sra d{d+1}W$.  Proposition~\ref{Lemma7} gives $\Theta_W=-[m-1]\Theta^d(d,j,T)$, and the fact that $[m]-[m-1]=({-}1)^{m+1}$ gives the result.

Now suppose $d=2$.  Then
\[
U(2,j,T)\tub13=\yo{\setxy\bxr1\hdbxyr{2.5}\bxr1\bxd2\setlength\xpos{0cm}\bxr2\hdbxyr{1.5}\bxr2\bxr3\bxd j
\setlength\xpos{0cm}\bxr3\hdbxyr{1.5}\bxr3\bxr4}
\]
\normalsize and Proposition~\ref{Lemma5} gives
\[
\psi_{2,1}\circ\Theta(2,j,T)=[m]\Theta^2(2,j,T)+\Theta_V,
\]
where $U(2,j,T)\sra 23V$.  Proposition~\ref{Lemma7} gives $\Theta_V=-[m-1]\Theta^2(2,j,T)$ plus a scalar multiple of $\Theta_W$, where
\[
W\tub23=\yo{\setxy\bxr2\hdbxyr{3.5}\bxr2\bxd3\setlength\xpos{0cm}\bxr3\hdbxyr{1.5}\bxr3\bxr4\bxr j}
\]
and $W\agree23U(2,j,T)$.  Since $[m]-[m-1]=({-}1)^{m+1}$, we just need to show that $\Theta_W=0$.
For $4\ls k\ls j-1$ we have
\[
W\tub kk=\yw{\setxy\bxr k\hdbxyr{1.5}\bxr k\bxr\kpo}.
\]
We apply Lemma \ref{aabdbbbc} in rows $k,k+1$, for $k=3,\dots,j-3$ in turn, and we find that $\Theta_W=\Theta_X$, where
\[
X\tub{j-2}{j-1}=\yw{\setxy\bxr\jmt\hdbxyr{1.5}\bxr\jmt\bxr\jmo\bxd j\setlength\xpos{0cm}\bxr\jmo\hdbxyr{2.5}\bxr\jmo\bxr j}.
\]
Now Proposition~\ref{Lemma7} gives $\Theta_X=0$, since we get a factor of $[2]=0$.

\item
The tableau $U(d+1,j,T)$ contains a \pp{} in row $1$, with the remaining \pp s in row $d+1$.  Proposition~\ref{Lemma5} yields
\[
\psi_{d,1}\circ\Theta(d+1,j,T)=({-}1)^m\Theta^d(d,j,T)+\Theta_W,
\]
where $U(d+1,j,T)\sra d{d+1}W$.  But $\Theta_W=0$ by Lemma \ref{toomany}, and we are done.
\item
In this case all the \dd s and \pp s in $U(i,j,T)$ lie in rows of length at most $m$; so by Proposition~\ref{Lemma5} and Lemma \ref{toomany} we have $\psi_{d,1}\circ\Theta(i,j,T)=0$.
\item
In this case
\[
U(2,j,T)\tub24=\yo{\setxy\bxr2\hdbxyr{1.5}\bxr2\bxr3\bxd j\setlength\xpos{0cm}\bxr3\hdbxyr{1.5}\bxr3\bxd4\setlength\xpos{0cm}\bxr4\hdbxyr{1.5}\bxr4\bxr5},
\]
and Proposition~\ref{Lemma5} gives
\[
\psi_{3,1}\circ\Theta(2,j,T)=[m]\Theta_V+\Theta_W,
\]
where $U(2,j,T)\sra 33V$ and $U(2,j,T)\sra34W$.  Proposition~\ref{Lemma7} gives $\Theta_W=-[m-1]\Theta_V$, so we just need to show that $\Theta_V=\Theta^3(2,3,T)$.

Applying Proposition~\ref{Lemma7} twice, we find that $\Theta_V=-\Theta_X$, where $X$ is obtained from $V$ by interchanging the $j$ in row $2$ with a $3$ in row $3$.  We can apply Lemma \ref{block} to $X$ (with $a=3$, $b=j$) and we obtain $\Theta_X=-\Theta^3(2,3,T)$.
\item
This is a simple application of Proposition~\ref{Lemma5} and Proposition~\ref{Lemma7}.
\item
The $3$s in $U(i,j,T)$ all appear in row $3$, so Proposition~\ref{Lemma5} gives $\psi_{2,1}\circ\Theta(i,j,T)=\Theta_V$, where $U(i,j,T)\sra23V$.  Applying Proposition~\ref{Lemma7}, this equates to $({-}1)^m\Theta_W$, where
\[
W\tub13=\yo{\setxy\bxr1\hdbxyr{1.5}\bxr1\bxr1\bxd i
\setlength\xpos{0cm}\bxr2\hdbxyr{1.5}\bxr2\bxr2\bxd2
\setlength\xpos{0cm}\bxr3\hdbxyr{1.5}\bxr3\bxr j}
\]
(and $W\agree13U(i,j,T)$).  For $k=4,\dots,i-1$ row $k$ of $W$ consists entirely of $k$s, so we can apply Proposition~\ref{Lemma7}(2) repeatedly to move the $j$ from row $3$ down to row $i-1$.  We also apply Proposition~\ref{Lemma7}(2) in rows $1$ and $2$, and we find that $\Theta_W=({-}1)^{i+1}\Theta_X$, where
\[
X\tub{i-1}{j-1}=\yw{
\setxy\bxr\imo\hdbxyr{1.5}\bxr\imo1\bxd j
\setlength\xpos{0cm}\bxr i\hdbxyr{1.5}\bxr i\bxd\ipo
\setlength\xpos{0cm}\bxr\ipo\hdbxyr{1.5}\bxr\ipo\bxd\ipt
\setlength\xpos{0cm}\addtolength\ypos{-.5cm}\vdbxy{1.5}\addtolength\xpos{2.5cm}\vdbxy{1.5}\addtolength\xpos{1cm}\vdbxy{1.5}
\addtolength\ypos{-1cm}\setlength\xpos{0cm}\bxr\jmo\hdbxyr{1.5}\bxr\jmo\bxr j}.
\]
\normalsize By Lemma \ref{block}, we have $\Theta_X=-\Theta^2(2,i,T)$, and we are done.

\item
This is a simple application of Proposition~\ref{Lemma5} and Proposition~\ref{Lemma7}.
\item
This is a simple application of Proposition~\ref{Lemma5} and Proposition~\ref{Lemma7}.
\item
Applying Proposition~\ref{Lemma5} and Proposition~\ref{Lemma7} gives $\psi_{1,1}\circ\Theta(i,j,T)=({-}1)^m\Theta_V$, where
\[
V\tub12=\yo{\setxy\bxr1\hdbxyr{3.5}\bxd1
\setlength\xpos{0cm}\bxr2\hdbxyr{1.5}\bxr2\bxr i\bxr j}
\]
and $V\agree12U(i,j,T)$.  Applying Lemma \ref{aacdbbb} gives
\[
\Theta_V=-\Theta_W-\Theta_X-\Theta_Y,
\]
where
\[
W\tub23=\yo{\setxy\bxr2\hdbxyr{2}\bxr2\bxr3\bxd j
\setlength\xpos{0cm}\bxr3\hdbxyr{2}\bxr3\bxd i},\qquad
X\tub23=\yo{\setxy\bxr2\hdbxyr{2}\bxr2\bxr3\bxd i
\setlength\xpos{0cm}\bxr3\hdbxyr{2}\bxr3\bxr j},\qquad
Y\tub23=\yo{\setxy\bxr2\hdbxyr{2.5}\bxr2\bxr3\bxd3
\setlength\xpos{0cm}\bxr3\hdbxyr{1.5}\bxr3\bxr i\bxr j}
\]
and $V\agree23W\agree23X\agree23Y$.  In particular, for $k=4,\dots,i-1$ the $k$th row of any of these tableaux consists entirely of $k$s.

For $W$, we can repeatedly apply Proposition~\ref{Lemma7}(2) to move the $i$ from row $3$ down to row $i-1$.  We get $\Theta_W=({-}1)^i\Theta^1(1,j,T)$.

We do the same for $X$ to reach a tableau in which the row $i-1$ has the form
\[
\yw{\setxy\bxr\imo\hdbxyr{1.5}\bxr\imo\bxd j}.
\]
We can apply Lemma \ref{block} to this tableau (with $a=i-1$, $b=j$) to obtain $\Theta_X=-({-}1)^i\Theta^1(1,i,T)$.

It remains to show that $\Theta_Y=0$.  Examining the tableau $Y\tub3{i-1}$, we find that there is a unique semistandard tableau with the same shape and content, so $\Theta_Y$ must equal a scalar multiple of $\Theta_Z$, where $Z\tub3{i-1}$ is this semistandard tableau and $Z\agree3{i-1}Y$.  Then
\[
Z\tub{i-1}{j-1}=\yw{
\setxy\bxr\imo\hdbxyr{1.5}\bxr\imo\bxr i\bxd j
\setlength\xpos{0cm}\bxr i\hdbxyr{1.5}\bxr i\bxr i\bxd\ipo
\setlength\xpos{0cm}\bxr\ipo\hdbxyr{1.5}\bxr\ipo\bxr\ipo\bxd\ipt
\setlength\xpos{0cm}\addtolength\ypos{-.5cm}
\vdbxy{1.5}\addtolength\xpos{2.5cm}
\vdbxy{1.5}\addtolength\xpos{1cm}\vdbxy{1.5}\addtolength\xpos{1cm}\vdbxy{1.5}
\addtolength\ypos{-1cm}\setlength\xpos{0cm}\bxr\jmo\hdbxyr{1.5}\bxr\jmo\bxr\jmo\bxr j
}.
\]
Applying Lemma \ref{aabdbbbc} repeatedly, we can move the $j$ from row $i-1$ down to row $j-2$; we obtain a tableau in which rows $j-2,j-1$ have the form
\[
\yw{
\setxy\bxr\jmt\hdbxyr{1.5}\bxr\jmt\bxr\jmo\bxd j
\setlength\xpos{0cm}\bxr\jmo\hdbxyr{2.5}\bxr\jmo\bxr j
}.
\]
\normalsize Now Proposition~\ref{Lemma7} tells us that the corresponding homomorphism is zero.\qedhere
\end{enumerate}

\renewcommand\theenumi{arabic{enumi}}
\renewcommand\labelenumi{\theenumi.}
\end{proof}

\subsubs{Rows $s$ to $s+f$}

\begin{propn}\label{middle}
Suppose $s\ls d\ls s+f-1$, and $i,j,T$ are as above.  Then $\psi_{d,1}\circ\Theta(i,j,T)=0$.
\end{propn}

\begin{proof}
We apply Proposition~\ref{Lemma5}.  Note that the \pp s in $U(i,j,T)$ occur in rows $d,d+1$.  Row $d$ contains $\mu_d-2$ \dd s and two \pp s, and all the remaining \dd s are in higher rows, so we get
\[
\psi_{d,1}\circ\Theta(i,j,T)=[\mu_{d}-1]\Theta_V+\Theta_W,
\]
where $U(i,j,T)\sra dd V$ and $U(i,j,T)\sra d{d+1}W$.  But Proposition~\ref{Lemma7} immediately gives $\Theta_W=-[\mu_d-3]\Theta_V$, are we are done.
\end{proof}

\begin{propn}\label{middle2}
Suppose $s+1\ls d\ls s+f$, and $i,j,T$ are as above.  Then $\psi_{d,2}\circ\Theta(i,j,T)=0$.
\end{propn}

\begin{proof}
This follows from Proposition~\ref{Lemma5} and Lemma \ref{toomany}.
\end{proof}

\subsubs{Rows $s+f$ to $s+f+g-1$}\label{firstofsix}

In the next few sections we prove the following, which will complete the proof of Proposition~\ref{mainhom} when $s,s'$ are odd.

\begin{propn}\label{bot1}
Suppose that $(i,j) \in \mathcal{I}$, that $s+f\ls d\ls s+f+g-2$ and that $t=1$ or $2$.  Then
\[
\psi_{d,t}\circ\left(\sum_{T\in\cala}\sgn(T)\Theta(i,j,T)\right)=0.
\]
\end{propn}

Note that the case $ t=2$, $d=s+f$ has already been covered in Proposition \ref{middle2}.

If $T \in \cala$, then the \pp s in $U(i,j,T)$ lie within rows $s+f,\dots,d+1$.  Given $s+f\ls k\ls d+1$, we write $a^T_k,b^T_k$ for the entries at the end of row $k$ of $T$, with $a_k^T\ls b_k^T$.  (We may just write $a_k,b_k$ if $T$ is understood.)

Now assume $d\ls s+f+g-3$ (the easier case $d=s+f+g-2$ is addressed below).   
The multiset $\lset{a_k,b_k}{s+f \ls k\ls d+1}$ contains two each of the integers $s+f+1,\dots,d+1$, together with two larger integers $a,b$. We partition $\cala$ according to these integers $a,b$: given $d+1<a\ls b$, we define $\cala^{a,b}$ to be the set of $T\in\cala$ such that $\lset{a_k,b_k}{s+f \ls k\ls d+1}$ includes the integers $a$ and $b$.  We prove the following refinement of Proposition \ref{bot1}.

\begin{propn}\label{bot1ab}
Suppose $(i,j) \in \mathcal{I}$ and that $s+f\ls d\ls s+f+g-3$, $ t=1$ or $2$ and $d+1<a\ls b$.  Then
\[
\psi_{d,t}\circ\left(\sum_{T\in\cala^{a,b}}\sgn(T)\Theta(i,j,T)\right)=0.
\]
\end{propn}

\newcounter{ct}
\setcounter{ct}0
\newcounter{ddpp}
\newcounter{dppa}
\newcounter{dppb}
\newcounter{ddpa}
\newcounter{dapp}
\newcounter{ddpb}
\newcounter{dbpp}
\newcounter{dapbe}
\newcounter{dbpae}
\newcounter{dapbu}
\newcounter{dbpau}
\newcounter{dapbd}
\newcounter{dbpad}
\newcounter{ddaae}
\newcounter{dpaae}
\newcounter{dapau}
\newcounter{dapad}
\newcounter{ddaau}
\newcounter{dpaau}
\newcounter{dpaad}
\newcounter{dapae}
\newcounter{ppaax}
\newcounter{ppaay}
\newcounter{papb}
\newcounter{pbpa}
\newcounter{abpp}
\newcounter{ppaa}
\newcounter{papa}
\newcounter{aapp}
\newcounter{ppx}
\newcounter{ppy}
\newcounter{dpe}
\newcounter{dde}
\newcounter{ddd}
\newcounter{dpu}
\newcounter{dpd}

We consider four different cases, according to whether $d>s+f$ and whether $a<b$.

\subsubs{The case $s+f+1\ls d\ls s+f+g-3,\ a<b$}

We start by defining some subsets of $\cala^{a,b}$:
{\allowdisplaybreaks\begin{align*}
\addtocounter{ct}1\setcounter{ddpp}{\value{ct}}\cala^{a,b}_{\arabic{ct}}&=\lset{T\in\cala^{a,b}}{(a_d,b_d,a_{d+1},b_{d+1})=(d,d,d+1,d+1)};\\
\addtocounter{ct}1\setcounter{dppa}{\value{ct}}\cala^{a,b}_{\arabic{ct}}&=\lset{T\in\cala^{a,b}}{(a_d,b_d,a_{d+1},b_{d+1})=(d,d+1,d+1,a)};\\
\addtocounter{ct}1\setcounter{dppb}{\value{ct}}\cala^{a,b}_{\arabic{ct}}&=\lset{T\in\cala^{a,b}}{(a_d,b_d,a_{d+1},b_{d+1})=(d,d+1,d+1,b)};\\
\addtocounter{ct}1\setcounter{ddpa}{\value{ct}}\cala^{a,b}_{\arabic{ct}}&=\lset{T\in\cala^{a,b}}{(a_d,b_d,a_{d+1},b_{d+1})=(d,d,d+1,a)};\\
\addtocounter{ct}1\setcounter{dapp}{\value{ct}}\cala^{a,b}_{\arabic{ct}}&=\lset{T\in\cala^{a,b}}{(a_d,b_d,a_{d+1},b_{d+1})=(d,a,d+1,d+1)};\\
\addtocounter{ct}1\setcounter{ddpb}{\value{ct}}\cala^{a,b}_{\arabic{ct}}&=\lset{T\in\cala^{a,b}}{(a_d,b_d,a_{d+1},b_{d+1})=(d,d,d+1,b)};\\
\addtocounter{ct}1\setcounter{dbpp}{\value{ct}}\cala^{a,b}_{\arabic{ct}}&=\lset{T\in\cala^{a,b}}{(a_d,b_d,a_{d+1},b_{d+1})=(d,b,d+1,d+1)};\\
\addtocounter{ct}1\setcounter{dapbe}{\value{ct}}\cala^{a,b}_{\arabic{ct}}&=\lset{T\in\cala^{a,b}}{(a_d,b_d,a_{d+1},b_{d+1})=(d,a,d+1,b),\ a_k=d\text{ for some }k<d};\\
\addtocounter{ct}1\setcounter{dbpae}{\value{ct}}\cala^{a,b}_{\arabic{ct}}&=\lset{T\in\cala^{a,b}}{(a_d,b_d,a_{d+1},b_{d+1})=(d,b,d+1,a),\ a_k=d\text{ for some }k<d};\\
\addtocounter{ct}1\setcounter{dapbu}{\value{ct}}\cala^{a,b}_{\arabic{ct}}&=\lset{T\in\cala^{a,b}}{(a_d,b_d,a_{d+1},b_{d+1})=(d,a,d+1,b),\ b_k=d,\ b_l=d+1\text{ for some }k<l<d};\\
\addtocounter{ct}1\setcounter{dbpau}{\value{ct}}\cala^{a,b}_{\arabic{ct}}&=\lset{T\in\cala^{a,b}}{(a_d,b_d,a_{d+1},b_{d+1})=(d,b,d+1,a),\ b_k=d,\ b_l=d+1\text{ for some }k<l<d};\\
\addtocounter{ct}1\setcounter{dapbd}{\value{ct}}\cala^{a,b}_{\arabic{ct}}&=\lset{T\in\cala^{a,b}}{(a_d,b_d,a_{d+1},b_{d+1})=(d,a,d+1,b),\ b_k=d+1,\ b_l=d\text{ for some }k<l<d};\\
\addtocounter{ct}1\setcounter{dbpad}{\value{ct}}\cala^{a,b}_{\arabic{ct}}&=\lset{T\in\cala^{a,b}}{(a_d,b_d,a_{d+1},b_{d+1})=(d,b,d+1,a),\ b_k=d+1,\ b_l=d\text{ for some }k<l<d}.
\end{align*}}

Using the definition of $\cala$, it is easy to check that these sets partition $\cala^{a,b}$.  (Note that because $a<b$, a tableau $T\in\cala^{a,b}$ cannot have $(a_{d+1},b_{d+1})=(a,b)$, because this would violate the splitting condition; similarly, we cannot have $\{a_d,b_d,a_{d+1},b_{d+1}\}=\{d+1,d+1,a,b\}$ as multisets.)  Now Proposition \ref{bot1ab} in this case will follow from the following two results.

\begin{propn}\label{mid1abprecise}
Suppose $i,j,d,a,b$ are as above with $a<b$.
\begin{enumerate}
\item\label{mid1ab0}
If $T\in\cala^{a,b}_{\arabic{ddpp}}\cup\cala^{a,b}_{\arabic{dppa}}\cup\cala^{a,b}_{\arabic{dppb}}$, then $\psi_{d,1}\circ\Theta(i,j,T)=0$.
\item\label{mid1abdddpa}
There is a bijection $T\mapsto T'$ from $\cala^{a,b}_{\arabic{ddpa}}$ to $\cala^{a,b}_{\arabic{dapp}}$ such that
\[
\psi_{d,1}\circ\left(\sgn(T)\Theta(i,j,T)+\sgn(T')\Theta(i,j,T')\right)=0\tag*{\textit{for each $T\in\cala^{a,b}_{\arabic{ddpa}}$.}}
\]
\item\label{mid1abdddpb}
There is a bijection $T\mapsto T'$ from $\cala^{a,b}_{\arabic{ddpb}}$ to $\cala^{a,b}_{\arabic{dbpp}}$ such that
\[
\psi_{d,1}\circ\left(\sgn(T)\Theta(i,j,T)+\sgn(T')\Theta(i,j,T')\right)=0\tag*{\textit{for each $T\in\cala^{a,b}_{\arabic{ddpb}}$.}}
\]
\item\label{mid1abdadpb}
There is a bijection $T\mapsto T'$ from $\cala^{a,b}_{\arabic{dapbe}}$ to $\cala^{a,b}_{\arabic{dbpae}}$ such that
\[
\psi_{d,1}\circ\left(\sgn(T)\Theta(i,j,T)+\sgn(T')\Theta(i,j,T')\right)=0\tag*{\textit{for each $T\in\cala^{a,b}_{\arabic{dapbe}}$.}}
\]
\item\label{mid1abfour}
There are bijections
\begin{alignat*}3
\cala^{a,b}_{\arabic{dapbu}}&\longrightarrow\cala^{a,b}_{\arabic{dbpau}}&\qquad\cala^{a,b}_{\arabic{dapbu}}&\longrightarrow\cala^{a,b}_{\arabic{dapbd}}&\qquad\cala^{a,b}_{\arabic{dapbu}}&\longrightarrow\cala^{a,b}_{\arabic{dbpad}}\\
T&\longmapsto T'&T&\longmapsto T''&T&\longmapsto T'''
\end{alignat*}
such that
\[
\psi_{d,1}\circ\left(\sgn(T)\Theta(i,j,T)+\sgn(T')\Theta(i,j,T')+\sgn(T)\Theta(i,j,T'')+\sgn(T')\Theta(i,j,T''')\right)=0\tag*{\textit{for each $T\in\cala^{a,b}_{\arabic{dapbu}}$.}}
\]
\end{enumerate}
\end{propn}

\begin{pfenum}
\item
If $T\in\cala^{a,b}_{\arabic{ddpp}}$, then we get $\psi_{d,1}\circ\Theta(i,j,T)=0$ by Proposition~\ref{Lemma5} and Lemma \ref{toomany}.

If $T\in\cala^{a,b}_{\arabic{dppa}}$ or $\cala^{a,b}_{\arabic{dppb}}$, then Proposition~\ref{Lemma5} gives
\[
\psi_{d,1}\circ\Theta(i,j,T)=[\mu_d]\Theta_V+\Theta_W,
\]
where $V\ars ddU(i,j,T)\sra d{d+1}W$.  But Proposition~\ref{Lemma7} gives $\Theta_W=-[\mu_d-2]\Theta_V$, so $\psi_{d,1}\circ\Theta(i,j,T)=0$.
\item
Given $T\in\cala^{a,b}_{\arabic{ddpa}}$, there is one \pp{} in a row $k<d$ of $U(i,j,T)$.  We define $T'$ by replacing this entry with \dd , and replacing $(a_d,b_d,a_{d+1},b_{d+1})$ with $(d,a,d+1,d+1)$.  It is easy to check that this really does define a bijection from $\cala^{a,b}_{\arabic{ddpa}}$ to $\cala^{a,b}_{\arabic{dapp}}$.  Moreover, it is clear that $\sgn(T')=\sgn(T)$.  Now consider $\psi_{d,1}\circ\Theta(i,j,T)$.  Applying Proposition~\ref{Lemma5}, we replace a \pp{} with a \dd{} in row $k$ or row $d+1$.  But in the latter case the resulting homomorphism is zero, by Lemma \ref{toomany}.  So we have $\psi_{d,1}\circ\Theta(i,j,T)=({-}1)^{\mu_d}\Theta_V$, where $U(i,j,T)\sra dkV$.  On the other hand, we have $\psi_{d,1}\circ\Theta(i,j,T')=\Theta_W$, where $U(i,j,T')\sra d{d+1}W$.  Applying Proposition~\ref{Lemma7} to $W$, we get $\Theta_W=({-}1)^{\mu_d-1}\Theta_V$; so $\psi_{d,1}\circ\left(\Theta(i,j,T)+\Theta(i,j,T')\right)=0$, as required.
\item
This is identical to (\ref{mid1abdddpa}), with the r\^oles of $a$ and $b$ interchanged.
\item
Given $T$, we define $T'$ simply by interchanging the \ba{} and \bb{} in rows $d$ and $d+1$ of $U(i,j,T)$.  It is very easy to see that this is a bijection, and that $\sgn(T)=\sgn(T')$.  Now we apply Proposition~\ref{Lemma5} to compute $\psi_{d,1}\circ\Theta(i,j,T)$; note that since we have $a_k=d$, we must have $b_k=d+1$.  Hence when we replace a \pp{} with a \dd{} in row $k$, we get a factor of $[2]=0$.  So we just have $\psi_{d,1}\circ\Theta(i,j,T)=\Theta_V$, where $U(i,j,T)\sra d{d+1}V$.  Similarly $\psi_{d,1}\circ\Theta(i,j,T')=\Theta_W$ where $U(i,j,T')\sra d{d+1}W$; applying Proposition~\ref{Lemma7} to $V$, we get $\Theta_V=({-}1)^{\mu_d-1}\Theta_X$, where $X$ is obtained by interchanging the \dd{} in row $d+1$ and the \ba{} in row $d$.  Similarly $\Theta_W=({-}1)^{\mu_d}\Theta_X$ (the difference in signs arising because $a<b$), and so $\psi_{d,1}\circ\left(\Theta(i,j,T)+\Theta(i,j,T')\right)=0$, as required.
\item
Given $T$, we obtain $T'$ by interchanging the \ba{} and \bb{} in rows $d,d+1$.  We obtain $T''$ by interchanging the $d,d+1$ in rows $k$ and $l$, and we obtain $T'''$ by doing both of these changes.  It is easy to see that these maps are bijections, and that $\sgn(T)=\sgn(T')=\sgn(T'')=\sgn(T''')$.  Proposition~\ref{Lemma5} gives
\begin{alignat*}2
\psi_{d,1}\circ\Theta(i,j,T)&=&({-}1)^{\mu_d-1}\Theta_V&+\Theta_{X_1}\\
\psi_{d,1}\circ\Theta(i,j,T')&=\ &({-}1)^{\mu_d-1}\Theta_W&+\Theta_{X_2}\\
\psi_{d,1}\circ\Theta(i,j,T'')&=&({-}1)^{\mu_d}\Theta_V&+\Theta_{Y_1}\\
\psi_{d,1}\circ\Theta(i,j,T''')&=&({-}1)^{\mu_d}\Theta_W&+\Theta_{Y_2},
\end{alignat*}
where
\[
U(i,j,T)\sra dlV\ars dkU(i,j,T''),\qquad U(i,j,T')\sra dlW\ars dkU(i,j,T''')
\]
and
\[
U(i,j,T)\sra d{d+1}X_1,\qquad U(i,j,T')\sra d{d+1}X_2,\qquad U(i,j,T'')\sra d{d+1}Y_1,\qquad U(i,j,T''')\sra d{d+1}Y_2.
\]
Using Proposition~\ref{Lemma7} as in (\ref{mid1abdadpb}) above, we get $\Theta_{X_1}=-\Theta_{X_2}$ and $\Theta_{Y_1}=-\Theta_{Y_2}$, so that
\[
\psi_{d,1}\circ\left(\Theta(i,j,T)+\Theta(i,j,T')+\Theta(i,j,T'')+\Theta(i,j,T''')\right)=0.\tag*{\qedhere}
\]
\end{pfenum}

\begin{propn}\label{mid2abprecise}
Suppose $i,j,d,a,b$ are as above with $a<b$.
\begin{enumerate}
\item\label{mid2ab0}
If $T\in\bigcup_{i=1}^9\cala^{a,b}_i$, then $\psi_{d,2}\circ\Theta(i,j,T)=0$.
\item\label{mid2abdadpb1}
There is a bijection $T\mapsto T'$ from $\cala^{a,b}_{\arabic{dapbu}}$ to $\cala^{a,b}_{\arabic{dbpau}}$ such that
\[
\psi_{d,2}\circ\left(\sgn(T)\Theta(i,j,T)+\sgn(T')\Theta(i,j,T')\right)=0\tag*{\textit{for each $T\in\cala^{a,b}_{\arabic{dapbu}}$.}}
\]
\item\label{mid2abdadpb2}
There is a bijection $T\mapsto T'$ from $\cala^{a,b}_{\arabic{dapbd}}$ to $\cala^{a,b}_{\arabic{dbpad}}$ such that
\[
\psi_{d,2}\circ\left(\sgn(T)\Theta(i,j,T)+\sgn(T')\Theta(i,j,T')\right)=0\tag*{\textit{for each $T\in\cala^{a,b}_{\arabic{dapbd}}$.}}
\]
\end{enumerate}
\end{propn}

\begin{pfenum}
\item\label{mid2abprecise3:abc:2aalast}
If $T\in\cala^{a,b}_i$ for $1\ls i\ls 7$, then the result follows by Proposition~\ref{Lemma5} and Lemma \ref{toomany}.  So suppose $T\in\cala^{a,b}_{\arabic{dapbe}}$ or $\cala^{a,b}_{\arabic{dbpae}}$, and consider applying Proposition~\ref{Lemma5}.  The \pp s in $U(i,j,T)$ lie in rows $k$ and $d+1$.  If we replace \pp{} with \dd{} in row $k$, then we get a factor of $[2]=0$; on the other hand, if we replace two \pp s with \dd s in row $d+1$, then the resulting homomorphism is zero by Lemma \ref{toomany}.
\item
The bijection $T\mapsto T'$ is the same as in Proposition \ref{mid1abprecise}(\ref{mid1abfour}).  Now consider applying Proposition~\ref{Lemma5}.  The \pp s in $U(i,j,T)$ all lie in row $d+1$ except for one, which lies in row $l$.  If we replace two of the \pp s in row $d+1$ with \dd s, then by Lemma \ref{toomany} the corresponding homomorphism is zero, so Proposition~\ref{Lemma5} gives
\[
\psi_{d,2}\circ\Theta(i,j,T)=({-}1)^{\mu_d-1}\Theta_X,
\]
where $U(i,j,T)\sra d{l,d+1}X$.  Proposition~\ref{Lemma5} similarly gives
\[
\psi_{d,2}\circ\Theta(i,j,T')=({-}1)^{\mu_d-1}\Theta_{X'}
\]
where $U(i,j,T')\sra d{l,d+1}X'$.  Now we apply Proposition~\ref{Lemma7} to $X$ and $X'$, in both cases moving the \dd{} from row $d+1$ to row $d$.  We find that $\Theta_X=-\Theta_{X'}$, and hence $\psi_{d,2}\circ\left(\Theta(i,j,T)+\Theta(i,j,T')\right)=0$, as required.
\item
This case is identical to the previous case, except that $({-}1)^{\mu_d-1}$ should be replaced with $({-}1)^{\mu_d}$.
\end{pfenum}

\subsubs{The case $s+f+1\ls d\ls s+f+g-3,\ a=b$}

Now we consider the case $a=b$.  The method is the same as in the last section, but we need to define some different subsets of $\cala^{a,a}$:
\setcounter{ct}0
{\allowdisplaybreaks
\begin{align*}
\addtocounter{ct}1\setcounter{ddpp}{\value{ct}}\cala^{a,a}_{\arabic{ddpp}}&=\lset{\vbox to 10pt{}T\in\cala^{a,a}}{(a_d,b_d,a_{d+1},b_{d+1})=(d,d,d+1,d+1)};\\
\addtocounter{ct}1\setcounter{dppa}{\value{ct}}\cala^{a,a}_{\arabic{dppa}}&=\lset{\vbox to 10pt{}T\in\cala^{a,a}}{(a_d,b_d,a_{d+1},b_{d+1})=(d,d+1,d+1,a)};\\
\addtocounter{ct}1\setcounter{ddpa}{\value{ct}}\cala^{a,a}_{\arabic{ddpa}}&=\lset{\vbox to 10pt{}T\in\cala^{a,a}}{(a_d,b_d,a_{d+1},b_{d+1})=(d,d,d+1,a)};\\
\addtocounter{ct}1\setcounter{dapp}{\value{ct}}\cala^{a,a}_{\arabic{dapp}}&=\lset{\vbox to 10pt{}T\in\cala^{a,a}}{(a_d,b_d,a_{d+1},b_{d+1})=(d,a,d+1,d+1)};\\
\addtocounter{ct}1\setcounter{ddaae}{\value{ct}}\cala^{a,a}_{\arabic{ddaae}}&=\lset{\vbox to 10pt{}T\in\cala^{a,a}}{(a_d,b_d,a_{d+1},b_{d+1})=(d,d,a,a),\ a_k=d+1\text{ for some }k<d};\\
\addtocounter{ct}1\setcounter{dpaae}{\value{ct}}\cala^{a,a}_{\arabic{dpaae}}&=\lset{\vbox to 10pt{}T\in\cala^{a,a}}{(a_d,b_d,a_{d+1},b_{d+1})=(d,d+1,a,a),\ a_k=d\text{ for some }k<d};\\
\addtocounter{ct}1\setcounter{dapau}{\value{ct}}\cala^{a,a}_{\arabic{dapau}}&=\lset{\vbox to 10pt{}T\in\cala^{a,a}}{(a_d,b_d,a_{d+1},b_{d+1})=(d,a,d+1,a),\ b_k=d,\ b_l=d+1\text{ for some }k<l<d};\\
\addtocounter{ct}1\setcounter{dapad}{\value{ct}}\cala^{a,a}_{\arabic{dapad}}&=\lset{\vbox to 10pt{}T\in\cala^{a,a}}{(a_d,b_d,a_{d+1},b_{d+1})=(d,a,d+1,a),\ b_k=d+1,\ b_l=d\text{ for some }k<l<d};\\
\addtocounter{ct}1\setcounter{ddaau}{\value{ct}}\cala^{a,a}_{\arabic{ddaau}}&=\lset{\vbox to 10pt{}T\in\cala^{a,a}}{(a_d,b_d,a_{d+1},b_{d+1})=(d,d,a,a),\ b_k=d+1,\ b_l=d+1\text{ for some }k<l<d};\\
\addtocounter{ct}1\setcounter{dpaau}{\value{ct}}\cala^{a,a}_{\arabic{dpaau}}&=\lset{\vbox to 10pt{}T\in\cala^{a,a}}{(a_d,b_d,a_{d+1},b_{d+1})=(d,d+1,a,a),\ b_k=d,\ b_l=d+1\text{ for some }k<l<d};\\
\addtocounter{ct}1\setcounter{dpaad}{\value{ct}}\cala^{a,a}_{\arabic{dpaad}}&=\lset{\vbox to 10pt{}T\in\cala^{a,a}}{(a_d,b_d,a_{d+1},b_{d+1})=(d,d+1,a,a),\ b_k=d+1,\ b_l=d\text{ for some }k<l<d};\\
\addtocounter{ct}1\setcounter{dapae}{\value{ct}}\cala^{a,a}_{\arabic{dapae}}&=\lset{\vbox to 10pt{}T\in\cala^{a,a}}{(a_d,b_d,a_{d+1},b_{d+1})=(d,a,d+1,a),\ a_k=d\text{ for some }k<d};\\
\addtocounter{ct}1\setcounter{ppaax}{\value{ct}}\cala^{a,a}_{\arabic{ppaax}}&=\lset{\vbox to 10pt{}T\in\cala^{a,a}}{(a_d,b_d,a_{d+1},b_{d+1})=(d+1,d+1,a,a)\text{ and either $a_{d-1}<d$ or $d=s+f+1$}};\\
\addtocounter{ct}1\setcounter{ppaay}{\value{ct}}\cala^{a,a}_{\arabic{ppaay}}&=\lset{\vbox to 10pt{}T\in\cala^{a,a}}{(a_{d-1},a_{d-1},a_d,b_d,a_{d+1},b_{d+1})=(d,d,d+1,d+1,a,a)\text{ and $d>s+f+1$}}.
\end{align*}}

Now Proposition \ref{bot1ab} in this case follows from the next two results.

\begin{propn}\label{mid1aaprecise}
Suppose $i,j,d,a$ are as above.
\begin{enumerate}
\item\label{mid1aa0}
If $T\in\cala^{a,a}_{\arabic{ddpp}}\cup\cala^{a,a}_{\arabic{dppa}}\cup\cala^{a,a}_{\arabic{dapae}}\cup\cala^{a,a}_{\arabic{ppaax}}\cup\cala^{a,a}_{\arabic{ppaay}}$, then $\psi_{d,1}\circ\Theta(i,j,T)=0$.
\item\label{mid1aadddpa}
There is a bijection $T\mapsto T'$ from $\cala^{a,a}_{\arabic{ddpa}}$ to $\cala^{a,a}_{\arabic{dapp}}$ such that
\[
\psi_{d,1}\circ\left(\sgn(T)\Theta(i,j,T)+\sgn(T')\Theta(i,j,T')\right)=0\tag*{\textit{for each $T\in\cala^{a,a}_{\arabic{ddpa}}$.}}
\]
\item\label{mid1aaddaa}
There is a bijection $T\mapsto T'$ from $\cala^{a,a}_{\arabic{ddaae}}$ to $\cala^{a,a}_{\arabic{dpaae}}$ such that
\[
\psi_{d,1}\circ\left(\sgn(T)\Theta(i,j,T)+\sgn(T')\Theta(i,j,T')\right)=0\tag*{\textit{for each $T\in\cala^{a,a}_{\arabic{ddaae}}$.}}
\]
\item\label{mid1aadadpa}
There is a bijection $T\mapsto T'$ from $\cala^{a,a}_{\arabic{dapau}}$ to $\cala^{a,a}_{\arabic{dapad}}$ such that
\[
\psi_{d,1}\circ\left(\sgn(T)\Theta(i,j,T)+\sgn(T')\Theta(i,j,T')\right)=0\tag*{\textit{for each $T\in\cala^{a,a}_{\arabic{dapau}}$.}}
\]
\item\label{mid1aathree}
There are bijections
\begin{alignat*}2
\cala^{a,a}_{\arabic{ddaau}}&\longrightarrow\cala^{a,a}_{\arabic{dpaau}}&\qquad\cala^{a,a}_{\arabic{ddaau}}&\longrightarrow\cala^{a,a}_{\arabic{dpaad}}\\
T&\longmapsto T'&T&\longmapsto T''
\end{alignat*}
such that
\[
\psi_{d,1}\circ\left(\sgn(T)\Theta(i,j,T)+\sgn(T')\Theta(i,j,T')+\sgn(T)\Theta(i,j,T'')\right)=0\tag*{\textit{for each $T\in\cala^{a,a}_{\arabic{ddaau}}$.}}
\]
\end{enumerate}
\end{propn}

\begin{pfenum}
\item
The cases where $T\in\cala^{a,a}_{\arabic{ddpp}}\cup\cala^{a,a}_{\arabic{dppa}}$ are dealt with as in Proposition \ref{mid1abprecise}(\ref{mid1ab0}).  In the case where $T\in\cala^{a,a}_{\arabic{dapae}}$, consider applying Proposition~\ref{Lemma5}.  If we replace a \pp{} with a \dd{} in row $k$, then we get a factor of $[2]$.  On the other hand, if we replace a \pp{} with a \dd{} in row $d+1$ and then apply Proposition~\ref{Lemma7} to move this \dd{} up to row $d$, we again get a factor of $[2]$.

In the case where $T\in\cala^{a,a}_{\arabic{ppaax}}$ or $\cala^{a,a}_{\arabic{ppaay}}$, we get
\[
\psi_{d,1}\circ\Theta(i,j,T)=[\mu_d-1]\Theta_V+\Theta_W
\]
with $V\ars ddU(i,j,T)\sra d{d+1}W$.  But Proposition~\ref{Lemma7} gives $\Theta_W=-[\mu_d-3]\Theta_W$, so that $\psi_{d,1}\circ\Theta(i,j,T)=0$.
\item
This is identical to cases (\ref{mid1abdddpa},\ref{mid1abdddpb}) in Proposition \ref{mid1abprecise}.
\item
Given $T$, we define $T'$ by replacing a \pp{} with a \dd{} in row $k$ of $U(i,j,T)$, and replacing a \dd{} with a \pp{} in row $d$.  Again, it easy to check that this gives a bijection from $\cala^{a,a}_{\arabic{ddaae}}$ to $\cala^{a,a}_{\arabic{dpaae}}$, and that $\sgn(T)=\sgn(T')$.

Consider applying Proposition~\ref{Lemma5} to compute $\psi_{d,1}\circ\Theta(i,j,T)$.  There are two \pp s in row $k$ of $U(i,j,T)$, and the remaining \pp s lie in row $d+1$.  If we replace a \pp{} with a \dd{} in row $d+1$, then by Lemma \ref{toomany} the resulting homomorphism is zero.  So we get $\psi_{d,1}\circ\Theta(i,j,T)=({-}1)^{\mu_d}\Theta_V$, where $U(i,j,T)\sra dkV$.

For $T'$, we get
\[
\psi_{d,1}\circ\Theta(i,j,T')=[\mu_d]\Theta_V+\Theta_W,
\]
where $V\ars ddU(i,j,T')\sra d{d+1}W$.  Applying Proposition~\ref{Lemma7}, we get $\Theta_W=-[\mu_d-3]\Theta_V$, and since $({-}1)^x+[x]-[x-3]=0$ for any $x$, we have $\psi_{d,1}\circ\left(\Theta(i,j,T)+\Theta(i,j,T')\right)=0$, as required.
\item
Given $T$, we define $T'$ by interchanging the \dd{} in row $k$ of $U(i,j,T)$ and the \pp{} in row $l$.

Applying Proposition~\ref{Lemma5}, we have
\begin{alignat*}2
\psi_{d,1}\circ\Theta(i,j,T)& =\ &({-}1)^{\mu_d-1}&\Theta_V+\Theta_W\\
\psi_{d,1}\circ\Theta(i,j,T')&=&({-}1)^{\mu_d}&\Theta_V+\Theta_X
\end{alignat*}
where
\[
U(i,j,T)\sra dlV\ars dkU(i,j,T')\qquad U(i,j,T)\sra d{d+1}W,\qquad U(i,j,T')\sra d{d+1}X.
\]
Applying Proposition~\ref{Lemma7}, we get $\Theta_W=\Theta_X=0$ (Proposition~\ref{Lemma7} gives a factor of $[2]$ in both cases), and so we have $\psi_{d,1}\circ\left(\Theta(i,j,T)+\Theta(i,j,T')\right)=0$.
\item
Given $T$, define $T'$ by replacing the last \dd{} in row $d$ of $U(i,j,T)$ with a \pp , and the \pp{} in row $k$ with a \dd .  Define $T''$ by replacing the last \dd{} in row $d$ of $U(i,j,T)$ with a \pp , and the \pp{} in row $l$ with a \dd .  Again, it is easy to see that we have bijections, and that $\sgn(T)=\sgn(T')=\sgn(T'')$.

Consider applying Proposition~\ref{Lemma5} to compute $\psi_{d,1}\circ\Theta(i,j,T)$.  If we replace a \pp{} with a \dd{} in row $d+1$, then by Lemma \ref{toomany} the resulting homomorphism is zero; so
\[
\psi_{d,1}\circ\Theta(i,j,T)=({-}1)^{\mu_d}(\Theta_V+\Theta_W),
\]
where $V\ars dkU(i,j,T)\sra dlW$.

For $T',T''$ we have
\begin{alignat*}4
\psi_{d,1}\circ\Theta(i,j,T')& =\ &({-}1)^{\mu_d-1}&\Theta_X&\ +\ [\mu_d]&\Theta_V&\ +\ &\Theta_Y\\
\psi_{d,1}\circ\Theta(i,j,T'')&=&({-}1)^{\mu_d}&\Theta_X&\ +\ [\mu_d]&\Theta_W&\ +\ &\Theta_Z,
\end{alignat*}
where
\begin{alignat*}3
U(i,j,T')&\sra dlX,&\qquad U(i,j,T')&\sra ddV,&\qquad U(i,j,T')&\sra d{d+1}Y,\\
U(i,j,T'')&\sra dkX,& U(i,j,T'')&\sra ddW,& U(i,j,T'')&\sra d{d+1}Z.
\end{alignat*}
Proposition~\ref{Lemma7} gives
\[
\Theta_Y=-[\mu_d-3]\Theta_V,\qquad \Theta_Z=-[\mu_d-3]\Theta_W,
\]
so that
\[
\psi_{d,1}\circ\left(\Theta(i,j,T)+\Theta(i,j,T')+\Theta(i,j,T'')\right)=0.\tag*{\qedhere}
\]
\end{pfenum}

\begin{propn}\label{mid2aaprecise}
Suppose $i,j,d,a$ are as above.
\begin{enumerate}
\item\label{mid2aa0}
If $T\in\cala^{a,a}_{\arabic{ddpp}}\cup\cala^{a,a}_{\arabic{dppa}}\cup\cala^{a,a}_{\arabic{ddpa}}\cup\cala^{a,a}_{\arabic{dapp}}\cup\cala^{a,a}_{\arabic{dpaae}}\cup\cala^{a,a}_{\arabic{dapau}}\cup\cala^{a,a}_{\arabic{dapad}}\cup\cala^{a,a}_{\arabic{dapae}}\cup\cala^{a,a}_{\arabic{ppaay}}$, then $\psi_{d,2}\circ\Theta(i,j,T)=0$.
\item\label{mid2aaddpaa}
There is a bijection $T\mapsto T'$ from $\cala^{a,a}_{\arabic{dpaau}}$ to $\cala^{a,a}_{\arabic{dpaad}}$ such that
\[
\psi_{d,2}\circ\left(\sgn(T)\Theta(i,j,T)+\sgn(T')\Theta(i,j,T')\right)=0\tag*{\textit{for each $T\in\cala^{a,a}_{\arabic{dpaau}}$.}}
\]
\item\label{mid2aaddaa}
There is a bijection $T\mapsto T'$ from $\cala^{a,a}_{\arabic{ddaae}}\cup\cala^{a,a}_{\arabic{ddaau}}$ to $\cala^{a,a}_{\arabic{ppaax}}$ such that
\[
\psi_{d,2}\circ\left(\sgn(T)\Theta(i,j,T)+\sgn(T')\Theta(i,j,T')\right)=0\tag*{\textit{for each $T\in\cala^{a,a}_{\arabic{ddaae}}\cup\cala^{a,a}_{\arabic{ddaau}}$.}}
\]
\end{enumerate}
\end{propn}

\begin{pfenum}
\item

If $T\in\cala^{a,a}_{\arabic{ddpp}}\cup \cala^{a,a}_{\arabic{dppa}}\cup \cala^{a,a}_{\arabic{ddpa}}\cup \cala^{a,a}_{\arabic{dapp}}\cup \cala^{a,a}_{\arabic{ppaay}}$, then $\psi_{d,2}\circ\Theta(i,j,T)=0$ by Proposition~\ref{Lemma5} and Lemma \ref{toomany}.

If $T\in\cala^{a,a}_{\arabic{dpaae}}$, then we must have $b_k=d+1$.  When we apply Proposition~\ref{Lemma5}, if we replace \pp{} with \dd{} in row $k$, then we get a factor of $[2]=0$.  On the other hand, if we replace two \pp s with \dd s in rows $d,d+1$, then the resulting homomorphism is zero by Lemma \ref{toomany}.

If $T\in\cala^{a,a}_{\arabic{dapau}}$, $\cala^{a,a}_{\arabic{dapad}}$ or $T\in\cala^{a,a}_{\arabic{dapae}}$, consider applying Proposition~\ref{Lemma5}.  If we replace two \pp s with \dd s in row $d+1$, then the resulting homomorphism is zero by Lemma \ref{toomany}; so we need only consider changing a single \pp{} into a \dd{} in row $d+1$.  Now when we apply Proposition~\ref{Lemma7} to the resulting tableau to move the \dd{} from row $d+1$ to row $d$, we get a factor of $[2]=0$.  And so we have $\psi_{d,2}\circ\Theta(i,j,T)=0$.
\item
Given $T$, we define $T'$ by exchanging the \dd{} in row $k$ with the \pp{} in row $l$.  Again, it is clear that this defines a bijection and that $\sgn(T)=\sgn(T')$.  Now consider applying Proposition~\ref{Lemma5} to $\Theta(i,j,T)$.  If we change replace two \pp s with \dd s in rows $d,d+1$, then the resulting homomorphism is zero by Lemma \ref{toomany}; so the only terms which can possibly be zero are those which involve replacing \pp{} with \dd{} in row $k$.  The same statement applies to $T'$, and it is then immediate from Proposition~\ref{Lemma5} that $\psi_{d,2}\circ\Theta(i,j,T)=-\psi_{d,2}\circ\Theta(i,j,T')$.
\item
Suppose $T\in\cala^{a,a}_{\arabic{ddaae}}$ or $\cala^{a,a}_{\arabic{ddaau}}$.  Because $T \in \cala$ we have either $d=s+f+1$ or $a_{d-1}<d$.  So if we define $T'$ by replacing the two \pp s above row $d$ with \dd s, and replacing the last two \dd s in row $d$ with \pp s, then $T'\in\cala^{a,a}_{\arabic{ppaax}}$.  This gives a bijection from $\cala^{a,a}_{\arabic{ddaae}}\cup\cala^{a,a}_{\arabic{ddaau}}$ to $\cala^{a,a}_{\arabic{ppaax}}$, and we claim that $\sgn(T)=-\sgn(T')$.  Clearly $T$ is split at row $d-s-f+1$ and not at any higher row, so $\sgn(T)=({-}1)^{d-s-f+1}$.  $T'$ is split at row $d-s-f$ and not at any higher row (because if there are any higher rows, then the first entry in row $d-s-f$ is less than $d-s-f$ by assumption), so $\sgn(T')=({-}1)^{d-s-f}$.

So all we need to do is show that $\psi_{d,2}\circ\Theta(i,j,T)=\psi_{d,2}\circ\Theta(i,j,T')$.  Using Proposition~\ref{Lemma5} and Lemma \ref{toomany}, we find that $\psi_{d,2}\circ\Theta(i,j,T)=\Theta_X$, where $X$ is obtained from $U(i,j,T)$ by replacing the two \pp s above row $d$ with \dd s.  On the other hand,
\[
\psi_{d,2}\circ\Theta(i,j,T')=\mbinoq{\mu_d}2\Theta_X+[\mu_d-1]\Theta_Y+\Theta_Z,
\]
where
\[
U(i,j,T')\sra d{d,d}X,\qquad U(i,j,T')\sra d{d,d+1}Y,\qquad U(i,j,T')\sra d{d+1,d+1}Z.
\]
Proposition~\ref{Lemma7} gives
\[
\Theta_Y=-[\mu_d-3]\Theta_X,\qquad \Theta_Z=-\mbinoq{\mu_d-3}2\Theta_X,
\]
and now the easy identity $\mbinoq x2-[x-1][x-3]-\mbinoq{x-3}2=1$ gives the result.
\end{pfenum}

\subsubs{The case $d=s+f,\ g\gs3,\  t=1,\ a<b$}

This case and the next are simpler than previous cases, so we spare the reader some of the details.  We define the following sets which partition $\cala^{a,b}$:
\setcounter{ct}0
\begin{align*}
\addtocounter{ct}1\setcounter{papb}{\value{ct}}\cala^{a,b}_{\arabic{papb}}&=\lset{T\in\cala^{a,b}}{(a_d,b_d,a_{d+1},b_{d+1})=(d+1,a,d+1,b)};\\
\addtocounter{ct}1\setcounter{pbpa}{\value{ct}}\cala^{a,b}_{\arabic{pbpa}}&=\lset{T\in\cala^{a,b}}{(a_d,b_d,a_{d+1},b_{d+1})=(d+1,b,d+1,a)};\\
\addtocounter{ct}1\setcounter{abpp}{\value{ct}}\cala^{a,b}_{\arabic{abpp}}&=\lset{T\in\cala^{a,b}}{(a_d,b_d,a_{d+1},b_{d+1})=(a,b,d+1,d+1)}.
\end{align*}

Now we have the following, from which Proposition \ref{bot1ab} follows in this case.

\begin{propn}\label{bot1abprecisesf}
Suppose $d=s+f$, and $i,j,a,b$ are as above with $a<b$.  There are bijections
\begin{alignat*}2
\cala^{a,b}_{\arabic{papb}}&\longrightarrow\cala^{a,b}_{\arabic{pbpa}}&\qquad\cala^{a,b}_{\arabic{papb}}&\longrightarrow\cala^{a,b}_{\arabic{abpp}}\\
T&\longmapsto T'&T&\longmapsto T''
\end{alignat*}
such that
\[
\psi_{d,1}\circ\left(\sgn(T)\Theta(i,j,T)+\sgn(T')\Theta(i,j,T')+\sgn(T)\Theta(i,j,T'')\right)=0\tag*{\textit{for each $T\in\cala^{a,b}_{\arabic{papb}}$.}}
\]
\end{propn}

\begin{proof}
The bijections in question are the obvious ones; we get $\sgn(T)=\sgn(T')=\sgn(T'')$.

Applying Proposition~\ref{Lemma5} followed by Proposition~\ref{Lemma7}, we have
\begin{alignat*}4
\psi_{d,1}\circ\Theta(i,j,T)&=&({-}1)^{\mu_d}&\Theta_V&&+&\ ({-}1)^{\mu_{d+1}-1}&\Theta_X,\\
\psi_{d,1}\circ\Theta(i,j,T')&=&({-}1)^{\mu_d}&\Theta_W&&+&({-}1)^{\mu_{d+1}}&\Theta_X,\\
\psi_{d,1}\circ\Theta(i,j,T'')&=&\ ({-}1)^{\mu_{d+1}}&\Theta_V&&+&({-}1)^{\mu_{d+1}}&\Theta_W,
\end{alignat*}
where
\[
U(i,j,T)\sra ddV,\qquad U(i,j,T')\sra ddW
\]
and $X$ has $d,d+1$ at the end of row $d$, and $a,b$ at the end of row $d+1$.  Since $s'$ is odd, $\mu_d,\mu_{d+1}$ have opposite parities, and so we get
\[
\psi_{d,1}\circ\left(\Theta(i,j,T)+\Theta(i,j,T')+\Theta(i,j,T'')\right)=0.\tag*{\qedhere}
\]
\end{proof}

\subsubs{The case $d=s+f,\ g\gs3,\  t=1,\ a=b$}

Here we define the following sets, which partition $\cala^{a,b}$:
\setcounter{ct}0
\begin{align*}
\addtocounter{ct}1\setcounter{ppaa}{\value{ct}}\cala^{a,a}_{\arabic{ppaa}}&=\lset{\vbox to 10pt{}T\in\cala^{a,a}}{(a_d,b_d,a_{d+1},b_{d+1})=(d+1,d+1,a,a)};\\
\addtocounter{ct}1\setcounter{papa}{\value{ct}}\cala^{a,a}_{\arabic{papa}}&=\lset{\vbox to 10pt{}T\in\cala^{a,a}}{(a_d,b_d,a_{d+1},b_{d+1})=(d+1,a,d+1,a)};\\
\addtocounter{ct}1\setcounter{aapp}{\value{ct}}\cala^{a,a}_{\arabic{aapp}}&=\lset{\vbox to 10pt{}T\in\cala^{a,a}}{(a_d,b_d,a_{d+1},b_{d+1})=(a,a,d+1,d+1)}.
\end{align*}

Now Proposition \ref{bot1ab} in this case is given by the following.

\begin{propn}\label{bot1aaprecisesf}
Suppose $d=s+f$, and $i,j,a$ are as above.
\begin{enumerate}
\item
If $T\in\cala^{a,a}_{\arabic{ppaa}}$, then $\psi_{d,1}\circ\Theta(i,j,T)=0$.
\item
There is a bijection $T\mapsto T'$ from $\cala^{a,b}_2$ to $\cala^{a,b}_3$ such that
\[
\psi_{d,1}\circ\left(\sgn(T)\Theta(i,j,T)+\sgn(T')\Theta(i,j,T')\right)=0\tag*{\textit{for each $T\in\cala^{a,b}_2$.}}
\]
\end{enumerate}
\end{propn}

\begin{pfenum}
\item
This is essentially the same as one of the cases in Proposition \ref{mid1aaprecise}(\ref{mid1aa0}).
\item
The bijection is the obvious one, and we obtain
\begin{align*}
\psi_{d,1}\circ\Theta(i,j,T)&=({-}1)^{\mu_d}\Theta_V,\\
\psi_{d,1}\circ\Theta(i,j,T')&=({-}1)^{\mu_{d+1}}\Theta_V,
\end{align*}
where $U(i,j,T)\sra ddV$.  Since $\mu_d$ and $\mu_{d+1}$ have opposite parities, the result follows.
\end{pfenum}

To finish off the proof of Proposition \ref{bot1}, we just need to consider $d=s+f+g-2$.

\subsubs{The case $d=s+f+g-2,\ g\gs3$}\label{lastcase}

In this case there are just two \pp s in $U(i,j,T)$, and we do not have the integers $a,b$. We define the following sets, which partition $\cala$:
\setcounter{ct}0
\begin{align*}
\addtocounter{ct}1\setcounter{ppx}{\value{ct}}\cala_{\arabic{ppx}}&=\lset{\vbox to 10pt{}T\in\cala}{(a_d,b_d)=(d+1,d+1)\text{ and either $a_{d-1}<d$ or $g=3$}};\\
\addtocounter{ct}1\setcounter{ppy}{\value{ct}}\cala_{\arabic{ppy}}&=\lset{\vbox to 10pt{}T\in\cala}{(a_{d-1},b_{d-1},a_d,b_d)=(d,d,d+1,d+1)\text{ and }g>3};\\
\addtocounter{ct}1\setcounter{dpe}{\value{ct}}\cala_{\arabic{dpe}}&=\lset{\vbox to 10pt{}T\in\cala}{(a_d,b_d)=(d,d+1),\ a_k=d\text{ for some }k<d};\\
\addtocounter{ct}1\setcounter{dde}{\value{ct}}\cala_{\arabic{dde}}&=\lset{\vbox to 10pt{}T\in\cala}{(a_d,b_d)=(d,d),\ a_k=d+1\text{ for some }k<d};\\
\addtocounter{ct}1\setcounter{ddd}{\value{ct}}\cala_{\arabic{ddd}}&=\lset{\vbox to 10pt{}T\in\cala}{(a_d,b_d)=(d,d),\ b_k=d+1,\ b_l=d+1\text{ for some }k<l<d};\\
\addtocounter{ct}1\setcounter{dpu}{\value{ct}}\cala_{\arabic{dpu}}&=\lset{\vbox to 10pt{}T\in\cala}{(a_d,b_d)=(d,d+1),\ b_k=d,\ b_l=d+1\text{ for some }k<l<d};\\
\addtocounter{ct}1\setcounter{dpd}{\value{ct}}\cala_{\arabic{dpd}}&=\lset{\vbox to 10pt{}T\in\cala}{(a_d,b_d)=(d,d+1),\ b_k=d+1,\ b_l=d\text{ for some }k<l<d}.
\end{align*}

\begin{propn}\label{verybot1precise}
Suppose $i,j$ are as above, and $d=s+f+g-2$.
\begin{enumerate}
\item
If $T\in\cala_{\arabic{ppx}}$ or $\cala_{\arabic{ppy}}$, then $\psi_{d,1}\circ\Theta(i,j,T)=0$.
\item
There is a bijection $T\mapsto T'$ from $\cala_{\arabic{dpe}}$ to $\cala_{\arabic{dde}}$ such that
\[
\psi_{d,1}\circ\left(\sgn(T)\Theta(i,j,T)+\sgn(T')\Theta(i,j,T')\right)=0\tag*{\textit{for each $T\in\cala_{\arabic{dpe}}$.}}
\]
\item
There are bijections
\begin{alignat*}2
\cala_{\arabic{ddd}}&\longrightarrow\cala_{\arabic{dpu}}&\qquad\cala_{\arabic{ddd}}&\longrightarrow\cala_{\arabic{dpd}}\\
T&\longmapsto T'&T&\longmapsto T''
\end{alignat*}
such that
\[
\psi_{d,1}\circ\left(\sgn(T)\Theta(i,j,T)+\sgn(T')\Theta(i,j,T')+\sgn(T)\Theta(i,j,T'')\right)=0\tag*{\textit{for each $T\in\cala_{\arabic{ddd}}$.}}
\]
\end{enumerate}
\end{propn}

\begin{pfenum}
\item
When we apply Proposition~\ref{Lemma5}, get a factor of $[2]=0$, since $\mu_d=3$.
\item
This is very similar to case (\ref{mid1aaddaa}) of Proposition \ref{mid1aaprecise}.  The difference here is that there is no tableau $W$; but in this case we have $\mu_d=3$, so that $({-}1)^{\mu_d}+[\mu_d]=0$ and the computation still works.
\item
This is very similar to case (\ref{mid1aathree}) of Proposition \ref{mid1aaprecise}.  In this case there are no tableaux $Y,Z$, but the computation goes through because $\mu_d=3$.
\end{pfenum}

\begin{propn}\label{verybot2precise}
Suppose $i,j$ are as above, and $d=s+f+g-2$.
\begin{enumerate}
\item
If $T\in\cala_{\arabic{ppy}}$ or $\cala_{\arabic{dpe}}$, then $\psi_{d,2}\circ\Theta(i,j,T)=0$.
\item
There is a bijection $T\mapsto T'$ from $\cala_{\arabic{dpu}}$ to $\cala_{\arabic{dpd}}$ such that
\[
\psi_{d,2}\circ\left(\sgn(T)\Theta(i,j,T)+\sgn(T')\Theta(i,j,T')\right)=0\tag*{\textit{for each $T\in\cala_{\arabic{dpu}}$.}}
\]
\item
There is a bijection $T\mapsto T'$ from $\cala_{\arabic{dde}}\cup\cala_{\arabic{ddd}}$ to $\cala_{\arabic{ppx}}$ such that
\[
\psi_{d,2}\circ\left(\sgn(T)\Theta(i,j,T)+\sgn(T')\Theta(i,j,T')\right)=0\tag*{\textit{for each $T\in\cala_{\arabic{ppx}}$.}}
\]
\end{enumerate}
\end{propn}

\begin{pfenum}
\item
The proof here is very similar to the proof of Proposition \ref{mid2aaprecise}(\ref{mid2aa0}).
\item
The proof here is the same as for Proposition \ref{mid2aaprecise}(\ref{mid2aaddpaa}).
\item
The proof here is very similar to the proof of Proposition \ref{mid2aaprecise}(\ref{mid2aaddaa}), but simpler, because there are no tableaux $Y,Z$.  The calculation still goes through because $\mu_d=3$.
\end{pfenum}

\subsubs{The case $d=s+f,\  t=1,\ g=2$}\label{lastofsix}


In this case, $\cala$ consists of a single tableau $T=\yo{\bxr1\bxr1}$, and for any $i,j$ the last row of $U(i,j,T)$ consists of $\mu_d-2$ \dd s followed by two \pp s.  Applying Proposition~\ref{Lemma5}, we get a factor of $[\la_d-1]=[2+s'-1]=0$, since $s'$ is odd; so $\psi_{d,1}\circ\Theta(i,j,T)=0$.

\subs{Proof of Proposition~\ref{mainhom} when $s=2$, $s'$ is even and $f=0$} \label{S:MHRedEven}

We now address the cases where $s=2$, $s'$ is even and $f=0$.  We let $\cala$ be the set of tableaux of shape and type $(2^{g-1})$ described in the last section.  Given $T\in\cala$, construct a $(2^g)$-tableau by increasing each entry by $2$ and adding a row $\yo{\bxr1\bxr2}$ at the top.  Let $U(T)$ be the corresponding usable $\mu$-tableau, and $\Theta_{(T)}$ the corresponding homomorphism.

\begin{eg}
Suppose $s'=4$ and $g=5$, and
\[
T=\yo{\bxr1\bxd4\bxl3\bxd1\bxr2\bxd2\bxl4\bxr3}.
\]
Then $T\in\cala$, and
\[
U(T)=
\yo{\setxy\bxr1\bxr1\bxr1\bxr1\bxr1\bxr1\bxr1\bxr1\bxr1\bxd2\bxl6\bxl3\bxl2
\bxl2\bxl2\bxl2\bxl2\bxl2\bxl2\bxd2\bxr3\bxr3\bxr3\bxr3\bxd5\addtolength\xpos{-1cm}\bxl4\bxl4\bxl4\bxd4\bxr5\bxr5\bxr6}.
\]

\end{eg}

Now we claim that the sum
\[
\sum_{T\in\cala}\sgn(T)\Theta_{(T)}
\]
gives a non-zero homomorphism from $S^\mu$ to $S^\la$, which completes the proof of Proposition \ref{mainhom}.  The proof of this is very similar to the proof in the previous case, in particular the parts in Sections \ref{firstofsix}--\ref{lastofsix}.  We leave the reader to check the details.

We remark that Proposition \ref{mainhom} seems to be true more generally: one can take $f>0$ and $s'\gs s$ even, and it seems to be the case that there is still a non-zero homomorphism from $S^\mu$ to $S^\la$.  But it does not seem quite so easy to write this homomorphism down.

\section*{Acknowledgements}
Some of the results presented in this paper were proved by the authors at MSRI Berkeley in 2008, during the concurrent programmes `Combinatorial representation theory' and `Representation theory of finite groups and related topics'.  Further work was carried out at the Mini-Workshop `Modular Representations of Symmetric Groups and Related Objects' in Oberwolfach in 2011.  
The authors would like to thank the organisers of all these programmes.  

The second author was supported by EPSRC grant EP/H052003.

\end{document}